\documentclass{scrartcl}
\usepackage{amsmath, amssymb, amsthm, marginnote, ltxcmds, adjustbox, stmaryrd, graphicx, bbold, stackengine, bbm}

\usepackage{mathtools, stackengine, changepage, ragged2e}

\usepackage{todonotes}

\definecolor{cite}{HTML}{11871E}
\definecolor{url}{HTML}{698996}
\definecolor{link}{HTML}{912F1B}

\usepackage[pdfencoding=unicode, colorlinks=true, linkcolor=link, citecolor=cite, urlcolor=url, linktocpage]{hyperref}

\usepackage[backend=biber, style=alphabetic, maxnames=10, maxalphanames=3, minalphanames=3]{biblatex}
\usepackage[capitalize]{cleveref}

\usepackage{tikz}
\usetikzlibrary{cd}
\usetikzlibrary{arrows, arrows.meta, positioning, calc}
\tikzcdset{arrow style=tikz, diagrams={>={Straight Barb[scale=0.8]}}}

\tikzstyle{arrow} = [-{Straight Barb[scale=0.8]}, line width=0.2mm]

\usepackage{enumitem}
{\end{enumerate}%
}

\usepackage[
	letterpaper,
	twoside=false,
	textheight=22cm,
	textwidth=14.4cm,
	marginparsep=0.75cm,
	marginparwidth=2.5cm,
	heightrounded,
	centering
]{geometry}

\usepackage{stmaryrd}
\Crefname{prop}{Proposition}{Propositions}
\Crefname{lem}{Lemma}{Lemmas}
\Crefname{cor}{Corollary}{Corollaries}
\Crefname{corAppx}{Corollary}{Corollaries}
\Crefname{thm}{Theorem}{Theorems}
\Crefname{thmAppx}{Theorem}{Theorems}
\Crefname{thmIntro}{Theorem}{Theorems}
\Crefname{alphThm}{Theorem}{Theorems}
\Crefname{alphProp}{Proposition}{Propositions}

\Crefname{defn}{Definition}{Definitions}
\Crefname{defnIntro}{Definition}{Definitions}
\Crefname{notation}{Notation}{Notations}
\Crefname{cons}{Construction}{Constructions}
\Crefname{rmk}{Remark}{Remarks}
\Crefname{obs}{Observation}{Observations}
\Crefname{trick}{Trick}{Tricks}
\Crefname{warning}{Warning}{Warnings}
\Crefname{conj}{Conjecture}{Conjectures}
\Crefname{assump}{Assumption}{Assumptions}
\Crefname{recollect}{Recollection}{Recollections}
\Crefname{terminology}{Terminology}{Terminologies}
\Crefname{condition}{Condition}{Conditions}
\Crefname{setting}{Setting}{Settings}

\Crefname{questionIntro}{Question}{Questions}
\Crefname{question}{Question}{Questions}
\Crefname{example}{Example}{Examples}

\Crefname{figure}{Figure}{Figures}

\crefformat{equation}{(#2#1#3)}
\crefformat{section}{\S#2#1#3}
\crefmultiformat{section}{\S\S#2#1#3}{and~#2#1#3}{, #2#1#3}{, and~#2#1#3}

\usepackage{aliascnt}
\newcommand{\newaliastheorem}[3]{%
  \newaliascnt{#1}{#2}%
  \newtheorem{#1}[#1]{#3}%
  \aliascntresetthe{#1}%
}

\newaliastheorem{thm}{subsubsection}{Theorem}
\newaliastheorem{prop}{subsubsection}{Proposition}
\newaliastheorem{lem}{subsubsection}{Lemma}
\newaliastheorem{cor}{subsubsection}{Corollary}

\newaliastheorem{thmAppx}{subsection}{Theorem}
\newaliastheorem{propAppx}{subsection}{Proposition}
\newaliastheorem{lemAppx}{subsection}{Lemma}
\newaliastheorem{corAppx}{subsection}{Corollary}

\newtheorem{alphThm}{Theorem}

\newcommand{\neutralize}[1]{\expandafter\let\csname c@#1\endcsname\count@}
\makeatother

\newaliastheorem{thmIntro}{subsection}{Theorem}

\theoremstyle{definition}
\newaliastheorem{defn}{subsubsection}{Definition}
\newaliastheorem{defnAppx}{subsection}{Definition}
\newaliastheorem{defnIntro}{subsection}{Definition}
\newaliastheorem{cons}{subsubsection}{Construction}
\newaliastheorem{nota}{subsubsection}{Notation}
\newaliastheorem{ass}{subsubsection}{Assumption}
\newaliastheorem{recollect}{subsubsection}{Recollections}
\newaliastheorem{terminology}{subsubsection}{Terminology}
\newaliastheorem{assump}{subsubsection}{Assumption}
\newaliastheorem{conj}{subsubsection}{Conjecture}
\newaliastheorem{idea}{subsubsection}{Idea}
\newaliastheorem{condition}{subsubsection}{Condition}
\newaliastheorem{question}{subsubsection}{Question}
\newaliastheorem{questionIntro}{subsection}{Question}
\newaliastheorem{setting}{subsubsection}{Setting}

\newaliastheorem{rmk}{subsubsection}{Remark}
\newaliastheorem{rmkAppx}{subsection}{Remark}
\newaliastheorem{obs}{subsubsection}{Observation}
\newaliastheorem{example}{subsubsection}{Example}
\newaliastheorem{fact}{subsubsection}{Fact}
\newaliastheorem{note}{subsubsection}{Note}
\newaliastheorem{warning}{subsubsection}{Warning}
\newaliastheorem{ideas}{subsubsection}{Ideas}

\newcommand{\target}{\mathrm{tgt}}

\newcommand{\posetCategory}{\mathrm{Poset}}
\newcommand{\poset}{\mathrm{Pos}}
\newcommand{\posetExample}{{L}}
\newcommand{\puncture}[1]{\mathring{#1}}
\newcommand{\initialObject}[1]{\widehat{#1}}
\newcommand{\subPoset}[1]{\breve{#1}}

\newcommand{\goodwillieTFunctor}{{T}}
\newcommand{\goodwillieApprox}{{P}}
\newcommand{\padding}{{C}}

\newcommand{\udl}[1]{\underline{#1}}

\newcommand{\distributiveInitialObject}{\varnothing}
\newcommand{\distributiveFinalObject}{\mathbbm{1}}

\newcommand{\cocartesianFunc}[1]{\underline{\mathrm{Fun}}^{{#1}_!}}
\newcommand{\cartesianFunc}{\underline{\mathrm{Fun}}^{\mathrm{cart}}}

\newcommand{\baseCat}{\mathcal{T}}

\newcommand{\excisiveCat}{\mathrm{Exc}}

\newcommand{\tcone}{^{\underline{\triangleleft}}}
\newcommand{\terminalTCat}{\underline{\ast}}
\newcommand{\tcocone}{^{\underline{\triangleright}}}

\newcommand{\basecat}{{\mathcal{T}}}

\newcommand{\totalCategory}{\mathrm{Total}}

\newcommand{\spc}{\mathcal{S}}

\newcommand{\canonical}{\mathrm{can}}

\newcommand{\constant}{\operatorname{const}}
\newcommand{\cat}{\mathrm{Cat}}

\newcommand{\presentable}{\mathrm{Pr}}

\newcommand{\J}{\mathcal{J}}
\newcommand{\I}{\mathcal{I}}

\newcommand{\sC}{{\mathcal C}}
\newcommand{\sD}{{\mathcal D}}

\newcommand{\D}{{\mathcal D}}

\newcommand{\op}{^{\mathrm{op}}}

\newcommand{\unstraighten}{\mathrm{UnStr}}

\newcommand{\M}{\mathcal{M}}

\newcommand{\E}{\mathcal{E}}

\newcommand{\orbit}{\mathcal{O}}

\newcommand{\spectra}{\mathrm{Sp}}
\newcommand{\finite}{\mathrm{Fin}}

\newcommand{\id}{\mathrm{id}}
\DeclareMathOperator{\map}{\mathrm{Map}}

\DeclareMathOperator{\func}{\mathrm{Fun}}

\newcommand{\unit}{\mathbbm{1}}

\def\colim{\qopname\relax m{colim}}

\DeclareMathOperator{\laxlim}{\mathrm{laxlim}}

\newcommand*\cocolon{%
        \nobreak
        \mskip6mu plus1mu
        \mathpunct{}%
        \nonscript
        \mkern-\thinmuskip
        {:}%
        \mskip2mu
        \relax
}

\usepackage{contour}
\usepackage{ulem}

\contourlength{0.5pt}

\newcommand{\myuline}[1]{%
  \uline{\phantom{#1}}%
  \llap{\contour{white}{#1}}%
}

\makeatletter
\newcommand*{\saved@myuline}{}
\let\saved@myuline\myuline

\newcommand*{\mathuline}{%
  \mathpalette{\math@myuline\saved@myuline}%
}
\newcommand*{\math@myuline}[3]{%
  \mbox{#1{$#2#3\m@th$}}%
}

\renewcommand*{\myuline}{%
  \relax  
  \ifmmode
    \expandafter\mathuline
  \else
    \expandafter\saved@myuline
  \fi
}
\makeatother

\title{Parametrised functor calculus: \\ excision, spheres, and semiadditivity}
\author{\textsc{Kaif Hilman}\thanks{Mathematisches Institut der Universität Bonn, Endenicher Allee 60, 53115 Bonn\\kaif@math.uni-bonn.de} 
\and \textsc{Sil Linskens}\thanks{Fakultät für Mathematik der Universit\"at Regensburg, Universit\"atsstraße 31, 93053 Regensburg\\sil.linskens@mathematiks.uni-regensburg.de}}
\date{\today}

\begin{document}
\maketitle

\begin{abstract}
We lay down the foundations of a theory of \textit{parametrised functor calculus}, generalising parts of the functor calculus of Goodwillie. We introduce the notion of excisable posets and develop a theory of excisive approximations in this context. As an application, we introduce two different excisable posets when parametrising over an atomic orbital category. By comparing the notions of excisiveness for these two posets, we relate the invertibility of certain spheres with Nardin's notion of parametrised semiadditivity, generalising Wirthm\"uller's classical result in equivariant homotopy theory for finite groups.
\end{abstract}

% \vspace{3mm}
% \begin{center}
%     \includegraphics[scale=0.2]{OpenArtstarryCubeVanGogh.jpg}\footnote{Image generated by OpenArt AI.}
% \end{center}

% \newpage

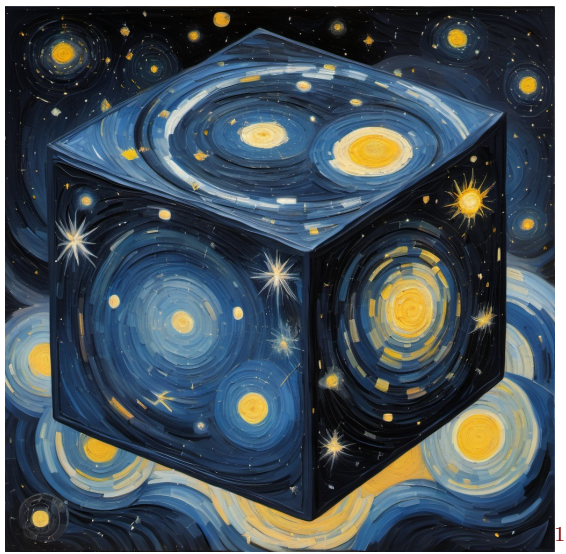
\tableofcontents

\footnotetext{\textup{2020} \textit{Mathematics Subject Classification}: \textup{18F50, 55P91, 55U35}} 

\newpage

\section{Introduction}

Multiplicative structures give rise to some of the deepest problems and methods in homotopy theory. As such, analysis of functors such as $X\mapsto X^{\otimes n}$ and its  variants are of great interest in this field. Fortunately, there is a far--reaching and systematic strategy to study such functors in the form of Goodwillie's \textit{functor calculus} \cite{goodwillieI,goodwillieII,goodwillieIII}, which is by now part of the standard repertoire of methods in higher algebra. Analogously, subtle multiplicative structures in the form of the norm functors of \cite{greenleesMayMU,HHR} often play a central role in nontrivial results coming from equivariant higher algebra as well as in their parametrised generalisations \cite{parametrisedIntroduction,expose1Elements,shahThesis,nardinThesis,shahPaperII,nardinShah}. One may thus ask for  a suitable generalisation of  calculus adapted to this setting in service of understanding the native polynomial phenomena therein. In this direction, Blumberg \cite{blumberg} and Dotto--Moi \cite{dottoMoi} have studied the notion of equivariant linearity, and this was then extended to a fully fledged theory of equivariant Goodwillie calculus by Dotto \cite{dottoExcision} in the setting of model categories. 

In this work, we make a first jab at a general theory of \textit{parametrised functor calculus} by developing a flexible framework of excision, inspired in part by Stoll's formulation in his study of functor calculus for non--cubes \cite{stoll}. The basic building blocks of our theory will be the notion of a \textit{parametrised excisable poset}, which we introduce in more detail below. As we explain there, we believe this definition captures the crucial structure of cubes used in classical Goodwillie calculus.
% Interestingly, the methods we apply to study these objects are rather ``Boolean'' in nature. We believe this extracts, morally, a core aspect of cubes used in functor calculus. 
Working in this generality leads to a presentation of the basic theory of calculus which only uses simple manipulations of left and right Kan extension along adjunctions which exist already at the level of posets\footnote{Our arguments often boil down to aligning various stars (and shrieks), used to denote Kan extensions. In this sense, our proofs are often pure ``astrology'', a phrase which we learnt from Shachar Carmeli.} and do not have to resort to more subtle cofinality arguments. The reason for this is in part aesthetic, but is also motivated by the more serious issue that many inductive proofs about cubes in ordinary Goodwillie calculus are not available to us when dealing with indexed products in the parametrised setting, and so we are forced to produce ``coordinate--free'' proofs for much of the elementary cubical facts. As such, the techniques employed in this article may be of independent interest. 

As a proof of concept, we employ our rudimentary theory to obtain a parametrised generalisation of the classical fact in equivariant homotopy theory that inverting the representation spheres yields the so-called Wirthm\"uller isomorphism. This will be achieved by constructing appropriate ``parametrised spheres'' using the cubes introduced in \cite{kaifNoncommMotives}. We consider this result to be slightly surprising, as it removes all of the geometry of the Pontryagin--Thom collapse which is crucial to traditional approaches to the Wirthm\"uller isomorphism.

% In some sense, ordinary calculus is a systematic study of nonlinear phenomena in mathematics via linear methods based upon the process of \textit{differentiation}. The key point here is that we are analysing the failure of $f$ from being a linear map by studying  how the difference $f(x+\delta)-f(x)$ varies as $\delta$ varies. One can then ask if this process categorifies well by replacing functions with functors, and if this is a fruitful pursuit. Quite miraculously, the answer to both questions is in the affirmative! This line of thought was first explored by Eilenberg and Mac Lane  \cite{eilenbergMacLanePolynomial} in their study of the homology of Eilenberg--Mac Lane spaces, and later on independently by Goodwillie \cite{goodwillieIII} in his comprehensive theory of {functor calculus}. Goodwillie's theory is now a standard repertoire in the methods of higher algebra and it is this theory that we seek to generalise to the parametrised context.

\paragraph{Excisable posets and excision}

% To motivate it, we recall that, as is formulated elegantly in functor calculus, the ``correct'' replacement of linear functions in the setting of functors between pointed categories is that of a \textit{reduced 1--excisive} functor, i.e. one that preserves the zero object and sends pushout squares to pullbacks. In other words,  linear functors are roughly speaking those that send cofibre sequences to fibre sequences, generalising, for instance, Eilenberg and Steenrod's notion of a homology theory. Thus, taking this as a cue, we should test the nonlinearity of a functor by measuring the failure of it to send cofibre sequences to fibre sequences.

Cubes constitute the grammar for Goodwillie's theory in that an ``$n$--excisive functor'' (this should be compared to polynomials of degree at most $n$ in ordinary algebra) is defined to be one which sends certain kinds of cocartesian $(n+1)$--cubes to cartesian ones. That this admits a rich theory mirroring many intuitions from ordinary algebra and at the same time captures many deep phenomena in homotopy theory is nothing short of miraculous. Even more amazingly, many of his methods are so robust that they are amenable to vast levels of generalisation. 

One such generalisation is the recent work of Stoll \cite{stoll} who extended Goodwillie's definition of excisiveness to special posets  he termed as \textit{shapes} which are not necessarily cubes. It was shown in \cite[Thm. A]{stoll} that the excisive approximations defined against shapes admit a formula akin to Goodwillie's original formula for $n$--excisive approximations. As will be clear to the reader familiar with Stoll's work, our setup and language in \cref{section:excision_for_posets} is heavily inspired by his. However, our methods are rather different: instead of using shapes as the building blocks for our theory, we opt for an alternative approach using what we term as \textit{excisable posets}. As mentioned, the advantage of this approach is that many proofs may be done very formally by exploiting adjunctions that exist already at the level of posets.

\vspace{1mm}

Before we explain what an excisable poset is, recall first that for a fixed category $\basecat$, a \textit{parametrised category} resp. \textit{functor} is an object resp. morphism in $\func(\basecat\op,\cat)$. This is the basic object  in Barwick--Dotto--Glasman--Nardin--Shah's theory of parametrised category theory \cite{parametrisedIntroduction,expose1Elements,shahThesis,shahPaperII,nardinExposeIV,nardinThesis}. A \textit{parametrised poset}\footnote{The consideration of such structured posets is not new: in the equivariant context, this has been explored for example in \cite{villaroelThesis,villaroelArxiv,JackowskiSlominska,mayStephanZakharevich} and also in \cite{dottoMoi,dottoExcision} for their equivariant calculus. } is then simply an object in the full subcategory $\func(\basecat\op,\posetCategory)\subseteq \func(\basecat\op,\cat)$, and we say that it is \textit{complementable} (cf. \cref{defn:parametrised_posets}) if it is distributive and for every object $x$, there exists another object $x^c$ with $x\vee x^c=\distributiveFinalObject$ where $\distributiveFinalObject$ is the final object of the poset. Examples of complementable posets include  those which are fibrewise given by cubes of various dimensions with (co)product--preserving transition functors.  

\begin{defn}[cf. \cref{defn:excisable_poset}]\label{defnIntro:excisable_posets}
    An \textit{excisable poset} is the datum of a complementable parametrised poset $\udl{\posetExample}$ together with a nonempty full subposet $\sigma\colon \subPoset{\udl{\posetExample}}\subseteq \udl{\posetExample}$ which is downward closed, i.e. for all $u\in \baseCat$, if $x\rightarrow y$ is a morphism in $\posetExample(u)$ where $y\in \subPoset{\posetExample}(u)$, then $x$ is also in $\subPoset{\posetExample}(u)$.
\end{defn}

\begin{defn}
    Let $\sigma\colon \subPoset{\udl{\posetExample}}\subseteq \udl{\posetExample}$ be an excisable poset and $F\colon \udl{\sC}\rightarrow\udl{\D}$ a parametrised functor. We say that $F$ is $\sigma$\textit{--excisive} if it sends $\udl{\posetExample}$--indexed diagrams in $\udl{\sC}$ left Kan extended from $\subPoset{\udl{\posetExample}}$ to a limit diagram in $\udl{\D}$.
\end{defn}

In the unparametrised context, these definitions recover Goodwillie's notion of $n$--excisiveness by using the excisable poset given by the inclusion of subsets of size at most 1 in the $(n+1)$--cube.

A basic insight of the present work is that the simple \cref{defnIntro:excisable_posets} affords a sufficiently rich environment to support the trappings of a theory of excisiveness while simultaneously being easy to work with. The complementability hypothesis plays a key role in allowing us to separate out orthogonal directions in the poset  by taking the product with an object $x\wedge-$ (see \cref{lem:complement_decompositions} for this procedure). What is more, since $x\wedge-$ is an idempotent procedure, extracting these orthogonal directions yield a Bousfield localisation and colocalisation on the poset, see \cref{cor:smashing_colocalisations}.Among other things, this allows us to generalise many of the proof steps due to Goodwillie (and in part, Rezk \cite{rezkGoodwillie}) very smoothly. They will generally be specific instances of formal manipulations of left and right Kan extensions along adjunctions. It also allows for a formal definition of the faces of an excisable poset, see \cref{subsec:face_decompositions}. This will be crucial in relating the different levels of excisiveness.

In particular, we generalise Goodwillie's formula for excisive approximations to our setting. In more detail, we prove that under some presentability hypotheses, the inclusion $\excisiveCat^n(\sC,\D)\subseteq \func(\sC,\D)$ of $n$--excisive functors into all functors admit a left adjoint $P_n$ for adjoint functor theorem reasons. However, what is remarkable is that, under a ``differentiability'' hypothesis (cf. \cite[Def. 6.1.1.6]{lurieHA}) on $\D$, Goodwillie's methods provides us an explicit formula for $\goodwillieApprox_n$ which is indispensable for both practical and theoretical purposes. For example, we have the formula
\[\goodwillieApprox_1F \simeq \colim_n\Omega^nF\Sigma^n.\]
Now, for an excisable poset $\sigma$, we give a similar explicit construction for $\goodwillieApprox_{\sigma}$ (cf. \cref{cons:excisive_approximation_functors}), and using a straightforward generalisation of the definition of differentiability to our context, our first main result may thus be stated  as follows:

\begin{alphThm}[Excisive approximations, cf. \cref{mainthm:excisive_approximation_left_adjoint}]\label{alphthm:excisive_approximation_left_adjoint}
    Let $\sigma\colon \subPoset{\udl{\posetExample}}\subseteq \udl{\posetExample}$ be an excisable poset. Let $\udl{\sC},\udl{\D}\in\cat_{\basecat}$ where $\udl{\sC}$ has $\sigma$--left Kan extensions and a final object and $\udl{\D}$ is $\sigma$--differentiable. Then the fully faithful inclusion $\udl{\excisiveCat}^{\sigma}(\udl{\sC},\udl{\D})\subseteq \udl{\func}(\udl{\sC},\udl{\D})$ admits a left adjoint given by $\goodwillieApprox_{\sigma}$.
\end{alphThm}

As mentioned above, it is the specific formula for $\goodwillieApprox_{\sigma}$ and not its mere existence that is the real content of the theorem. This is a point which we will crucially exploit to obtain our \cref{alphthm:spherically_faithful_implies_semiadditivity}.

\paragraph{Cubes, spheres, and semiadditivity}

Having set up some basic theory, one is then led to ask if this relates to or solves any problem that a priori has nothing to do with parametrised functor calculus. This is perhaps the light in which our second set of main results should be viewed, the context of which we explain now.

The following is one of the classical foundational results in equivariant stable homotopy theory, first proved by Wirthm\"uller\footnote{Even though Wirthm\"uller proved this for all compact Lie groups, we only give the case of finite groups since it is easier to state and is the only level of generality that is pertinent to this article.} \cite{wirthmuller}  at the level of homology theories. 

\begin{thm}\label{thmIntro:wirthmuller_iso}
    Let $G$ be a finite group and $S^{\rho_G}$ the regular representation sphere. Then writing $\spc_{G,*}\coloneqq \func(\orbit_G\op,\spc_*)$ for the category of pointed genuine $G$--spaces, the orbits $G/H_+\in\spc_{G,*}$ become self--dual in $\spectra_G\coloneqq\spc_{G,*}[(S^{\rho_G})^{-1}]$.  
\end{thm}
This self--duality in $\spectra_G$ is of fundamental importance in the methods of equivariant stable homotopy theory. In fact, such self--duality admits a generalisation in the form of \textit{parametrised semiadditivity} due to Nardin \cite{nardinExposeIV,nardinThesis} in any context parametrised over so--called ``atomic orbital'' base categories, see \textit{loc. cit.} Using this, Nardin also defined the notion of \textit{parametrised stability} for parametrised categories as being fibrewise stable and satisfying the aforementioned semiadditivity.

However, the relationship between  representation spheres and the orbits turns out to be quite  mysterious, at least from the categorical point of view. Indeed, the standard proof for \cref{thmIntro:wirthmuller_iso}  uses   equivariant smooth manifold theory via the Pontryagin--Thom collapse map to produce the duality datum witnessing the orbits as being self--dual. This approach is unfortunately not amenable to generalisation to the abstract parametrised setting where any kind of geometric argument is often not at our disposal. Moreover, unlike in the nonequivariant context where stability on a pointed category with sufficient colimits is tantamount to the suspension functor  being an equivalence, Nardin's definition of parametrised stability seems to ``separate out'' two distinct conditions. In light of this, one is then naturally led to ask the following:

\begin{question}\label{question:alternative_characterisation_stability}
    Is there an alternative characterisation of parametrised stability in terms of the invertibility of some appropriate  spheres? If so, what should these spheres be?
\end{question}

To proceed, we turn to Dotto--Moi for guidance: in their \cite[Thm. 3.26]{dottoMoi},  \cref{thmIntro:wirthmuller_iso} was proved essentially by way of a precursor to equivariant Goodwillie calculus, which was later  fleshed out  into a full theory by Dotto in \cite{dottoExcision}. Inspired by their methods, we specialise the axiomatic theory of excision above to a special family of excisable posets given by the \textit{singleton inclusions} in \textit{parametrised cubes} as introduced in \cite[\textsection 3]{kaifNoncommMotives} (cf. \cref{cons:singletonInclusion}) whose existence is a special feature of base categories that are  atomic orbital. Here, by a parametrised cube, we simply mean an indexed product of $\Delta^1$ (cf. \cref{defn:parametrised_cubes}). The role of these diagrams, facilitated by the key cubical inputs of \cref{prop:productsOfStronglyCocartesianCubes,keylem:excisive_implies_semiadditivity_induction}, is that they encode parametrised semiadditivity in the form of parametrised cubes, which in turn are primed for calculus--theoretic manipulations. We term the excisiveness with respect to these excisable posets as \textit{singleton excision.}

On the other hand, one can consider the excisable poset structure given by these parametrised cubes and the subposet away from the final object. Using this, we may construct certain \textit{parametrised spheres} (cf. \cref{subsection:parametrised_spheres}), supplying a replacement of the regular representation sphere in this general context. Excisiveness against such excisable posets are termed, correspondingly, as \textit{spherical excision.} Furthermore, we say that a pointed parametrised category $\udl{\sC}$ is \textit{spherically invertible} (cf. \cref{defn:spherically_faithful_category}) if the suspension functors along all parametrised cubes is an equivalence.  By comparing singleton and spherical excisions, we then employ a ``length induction''  and use \cref{alphthm:excisive_approximation_left_adjoint} to provide the following result:

\begin{alphThm}[cf. \cref{mainthm:spherically_faithful_implies_semiadditivity}]\label{alphthm:spherically_faithful_implies_semiadditivity}
Suppose $\basecat$ is a locally short atomic orbital category (cf. \cref{defn:locally_short_categories}). Let $\underline{\sC}$  be a pointed $\basecat$--category which is  parametrised bicomplete. If $\underline{\sC}$  is spherically invertible, then it is $\baseCat$--semiadditive.
\end{alphThm}

In fact, the theorem is proved using the weaker hypothesis of \textit{spherical faithfulness} (\cref{defn:spherically_faithful_category}), and to lend a convincing scope to the result, we supply plenty of examples of locally short atomic orbital categories in \cref{example:atomic_orbital}. This theorem may be viewed as a vast generalisation  of the classical result above (and thus also of \cite[Thm. 3.26]{dottoMoi}). Perhaps more philosophically, it shows that such results are not endemic just to the equivariant setting for finite groups and that they have a distinctly algebraic or categorical flavour independent of any geometric arguments. It bears mentioning, however, that our framework and results do \textit{not} apply to equivariance with respect to general compact Lie groups. Currently it seems that there is something irreducibly geometric at this level of generality. In this sense, \cref{alphthm:spherically_faithful_implies_semiadditivity} may be seen as a continuation in the development of our understanding of atomic orbital categories.

\vspace{1mm}

Under some additional hypotheses, we use \cref{alphthm:spherically_faithful_implies_semiadditivity} to give a new characterisation of parametrised stability, providing an answer to  \cref{question:alternative_characterisation_stability}.

\begin{alphThm}[cf. \Cref{mainThm:sphere_invertible_equivalent_to_parametrised_stable}]\label{alphThm:sphere_invertible_equivalent_to_parametrised_stable}
Suppose $\basecat$ is a locally short atomic orbital category such that $\myuline{\spectra}$ admits a $\basecat$-distributive symmetric monoidal structure.  Let $\udl{\sC}$ be a pointed parametrised presentable $\basecat$--category. Then $\udl{\sC}$ is parametrised stable if and only if it is spherically invertible. 
\end{alphThm}

In the statement of the theorem we must suppose that $\myuline{\spectra}$, the initial presentable parametrised stable category under $\myuline{\spc}$, admits a \textit{$\basecat$--distributive symmetric monoidal structure}. This is a generalisation of the structure of Hill--Hopkins--Ravenel norms to general $\basecat$. In particular $\myuline{\spectra}$ admits this structure when $\basecat$ is the orbit category $\mathcal{O}_G$ of a finite group $G$. As is explained in \cref{rmk:explanation_of_nardin_issues} this assumption is in fact necessary for the conclusion of the theorem to hold. We consider the question of whether it is satisfied for all $\basecat$ to be open, see \cref{rmk:explanation_of_nardin_issues} for a discussion of this point as well. 

By another application of the formula of excisive approximation furnished by \cref{alphthm:excisive_approximation_left_adjoint}, we also complement the preceding theorem by showing that, in some sense, there is just one notion of parametrised linearity in the stable setting.

\begin{alphThm}[Precise version in \cref{thm:characterisations_of_linearity}]
    Let $\baseCat$ be a locally short atomic orbital category, $\underline{\sC}$ a pointed parametrised bicomplete category, $\underline{\D}$ a spherically invertible parametrised presentable  category, and $F\colon \underline{\sC}\rightarrow\underline{\D}$ a reduced and filtered colimit preserving functor. Then $F$ preserving all parametrised colimits is equivalent to $F$ being spherically excisive, and this is in turn equivalent to $F$ being singleton excisive.
\end{alphThm}

\paragraph*{Structural overview}

We develop the foundations of parametrised excisiveness in \cref{section:excision_for_posets} by formulating the notion of an excisable poset and proving the fundamental \cref{alphthm:excisive_approximation_left_adjoint}. In the next \cref{section:parametrised_cubical_excision}, we specialise the preceding theory to atomic orbital base categories to prove  \cref{alphthm:spherically_faithful_implies_semiadditivity} and \cref{alphThm:sphere_invertible_equivalent_to_parametrised_stable} by introducing the parametrised spheres. In \cref{section:decomposition_of_colimits}, we record a parametrised enhancement of the decomposition results of \cite{horevYanovski}.

\subsection{Prerequisites, notations, and conventions}\label{subsec:notations}
This article is written in the formalism of parametrised higher categories as introduced and developed in \cite{parametrisedIntroduction,expose1Elements,nardinExposeIV,nardinShah,shahThesis,shahPaperII}. For a convenient one--stop reference on this, see also for example the survey given in \cite[\textsection 2-4]{kaifPresentable}. Most of our notations will be based on this latter source. Since they will play such a prevalent role, we highlight the following notations so as to prevent any ambiguity. 

Let  $\basecat$ be a base category. We shall write $\cat_{\basecat}$ for the category $\func(\basecat\op,\cat)$ of $\basecat$--categories and $\udl{\cat}$ for the $\basecat$--category of $\basecat$--categories. For a morphism $w\colon W \rightarrow V$ in $\finite_{\basecat}$, we write the associated induction--restriction--coinduction adjunctions as
\begin{center}
    \begin{tikzcd}
        \cat_{\basecat_{/V}}\ar[rr, "w^*"] & & \cat_{\basecat_{/W}}. \ar[ll, bend right = 30, "w_!"' description] \ar[ll, bend left = 26, "w_*" description]
    \end{tikzcd}
\end{center}
For a $\basecat$--category $\udl{\sC}\in\cat_{\basecat}$ and an object $V\in\finite_{\basecat}$, we write $V^*\udl{\sC}$ for the associated $\basecat_{/V}$--category obtained by restriction along the functor $\basecat_{/V}\rightarrow \basecat$.

Moreover, letting $\udl{\sC}$ be a $\basecat$--category, by an object $x\in \udl{\sC}$, we shall mean an implicit choice of $V\in\basecat$ and an object $x$  in the nonparametrised category $\udl{\sC}(V)\simeq  (V^*\udl{\sC})(V)\in\cat $ given by evaluating $\udl{\sC}$  at $V\in\basecat$. As such, when we associate to such an object a functor, for example, the product functor $x\times-\colon \udl{\sC}\rightarrow\udl{\sC}$, we shall mean that there is an implicit basechange to $\basecat_{/V}$-categories, so that the functor is really the functor $x\times-\colon V^*\udl{\sC}\rightarrow V^*\udl{\sC}$. We have adopted this convention so as not to clutter the statements with too many symbols.

\subsection{Acknowledgements}
We thank Tobias Barthel, Bastiaan Cnossen, Jesper Grodal, Markus Land, and  Maxime Ramzi for many useful and enlightening discussions on this subject. A special thanks goes to Emanuele Dotto for his continued interest and encouragements on this project: the great influence of his articles \cite{dottoMoi,dottoExcision} on the present work should be clear to the initiated. Lastly, we would also like to thank the anonymous referee for a careful and thorough reading of the paper and for many suggestions which have surely improved the document.

This work is a wide--reaching extension of the work done by KH as part of his PhD thesis \cite[Ch. 6]{kaifThesis}, and so he is grateful for the support of the Danish National Research Foundation through the
Copenhagen Centre for Geometry and Topology (DNRF151). KH is also supported by the European Research Council (ERC) under Horizon Europe (grant
No. 101042990) and would like to thank the Max Planck Institute for Mathematics (MPIM) in Bonn for its hospitality. SL would like thank the Mathematisches Institut der Universität Bonn, where much of the work on this project was undertaken. SL is currently a member of the faculty of Mathematics of Universit\"at Regensburg.

\section{Parametrised excision for posets}\label{section:excision_for_posets}

We introduce and prove  basic properties of the notions of excisable posets in \cref{subsec:excisable_posets} and of their associated faces in \cref{subsec:face_decompositions}; we then apply these preliminaries  in \cref{subsec:excision} to develop the concept of parametrised excision and prove \cref{alphthm:excisive_approximation_left_adjoint}. To begin, let us recall the notations an conventions from \cref{subsec:notations}.

\subsection{Excisable posets}\label{subsec:excisable_posets}

We begin by defining some terms and recording some easy observations on posets.

\begin{defn}\label{defn:parametrised_posets}
    A \textit{parametrised poset} is an object in $\posetCategory_{\baseCat}\coloneqq\func(\baseCat\op,\posetCategory)\subseteq\cat_{\baseCat}$.
    \begin{itemize}
        \item A parametrised poset is said to be \textit{distributive} if it admits parametrised fibrewise\footnote{That is, the $\basecat$--category admits finite (co)products levelwise and the restriction functors between different levels preserve these.} finite products and coproducts which we write as $\wedge$ and $\vee$ respectively, and such that the canonical map $(a\wedge c)\vee(b\wedge c)\rightarrow (a\vee b)\wedge c$ is an equivalence. We will write $\distributiveInitialObject$ and $\distributiveFinalObject$ for the initial and final objects, respectively.
        \item If furthermore, for every object $a$, there exists another object $a^c$ such that $a\vee a^c\simeq \distributiveFinalObject$ and $a\wedge a^c\simeq \distributiveInitialObject$, then we say that the parametrised poset is \textit{complementable}.
    \end{itemize}
    When the context is clear, we will often omit to mention the adjective ``parametrised'' and just refer to these objects as posets.
\end{defn}

\begin{rmk}
For simplicity we have phrased the definition above for parametrised categories over a small category $\basecat$. However, all the results in this section generalise immediately to the context of internal categories in any topos $\mathcal{B}$ in the sense of \cite{Martini2022Yoneda,MartiniWolf2024}.
\end{rmk}

\begin{rmk}
    In posets, since the space of morphisms between objects are either empty or contractible, we see that coproducts also compute pushouts, and similarly, products compute pullbacks. Hence, for example when $S$ is a finite set, in the poset of subsets $\poset(S)$, coproducts are given by unions of subsets and products are given by intersections of subsets.
\end{rmk}

\begin{rmk}\label{rmk:smashing_with_initial_object}
    For posets, we always have that $x\wedge \distributiveInitialObject = \distributiveInitialObject$ since the existence of the map $x\wedge \distributiveInitialObject\rightarrow \distributiveInitialObject$ ensures that the mapping space out of $x\wedge \distributiveInitialObject$ is nonempty, and so contractible.
\end{rmk}

\begin{lem}[Idempotence of poset (co)products]\label{lem:idempotence_poset_(co)products}
    Let $\udl{\posetExample}\in\posetCategory_{\baseCat}$. Then for any $x\in\udl{\posetExample}$, the fibrewise product $x\wedge x$ and fibrewise coproduct $x\vee x$ exist, and they are given by $x\wedge x \simeq x$ and $x\vee x \simeq x$.
\end{lem}
\begin{proof}
    We only do the case of coproducts, since the other is dual. Let $y\in\udl{\posetExample}$. If the coproduct $x\vee x$ exists, then it must participate in the equivalences
    \[\map_{\udl{\posetExample}}(x\vee x,y)\simeq \map_{\udl{\posetExample}}(x,y)\times \map_{\udl{\posetExample}}(x,y)\simeq\map_{\udl{\posetExample}}(x,y) \]
    where the second equivalence is since $\varnothing\times \varnothing\simeq \varnothing$ and $\ast\times \ast \simeq \ast$ in spaces. Thus, we see that $x$ satisfies the universal property of the coproduct.
\end{proof}

\begin{lem}[Saturation of (co)products]\label{lem:saturation_of_(co)products}
    Let $\udl{\posetExample}\in\posetCategory_{\baseCat}$. Then for any $x,a\in \udl{\posetExample}$, $x\vee(x\wedge a)\simeq x$ and $x\wedge(x\vee a)\simeq x$, if these (co)products exist.
\end{lem}
\begin{proof}
    We only do the first case, since the second is dual. Let $y\in\udl{\posetExample}$ and consider
    \[\map_{\udl{\posetExample}}(x\vee(x\wedge a),y)\simeq \map_{\udl{\posetExample}}(x,y)\times \map_{\udl{\posetExample}}(x\wedge a,y) \longrightarrow \map_{\udl{\posetExample}}(x,y).\] We argue that this map is an equivalence. If $\map_{\udl{\posetExample}}(x,y)\simeq \varnothing$, then both sides will be $\varnothing$ in spaces, and so is an equivalence in this case. If $\map_{\udl{\posetExample}}\simeq \ast$, then there exists a map $x\rightarrow y$, and so by precomposition with $x\wedge a\rightarrow x$, we also see that $\map_{\udl{\posetExample}}(x\wedge a,y)\simeq \ast$. Thus, the map of interest is an equivalence again.
\end{proof}

We next work out several elementary properties of complementable posets before coming to the definition of an excisable poset, which will play a central role throughout the article. The complementability assumption will allow us to decompose our posets as a product of two smaller complementable posets in very pleasant ways.

\begin{cons}[Smashing subposet]\label{cons:smashing_subposet}
    Let $\udl{\posetExample}$ be a complementable parametrised poset and $x\in \udl{\posetExample}$. Write $\udl{\posetExample}_{x}$ for the full subcategory of $\udl{\posetExample}$ spanned by the essential image of the idempotent functor $x\wedge-\colon \udl{\posetExample}\rightarrow \udl{\posetExample}$ from \cref{lem:idempotence_poset_(co)products}. The poset $\udl{\posetExample}_x$ clearly has initial object $\distributiveInitialObject$ since $x\wedge\distributiveInitialObject\simeq \distributiveInitialObject$, and by \cref{lem:saturation_of_(co)products}, it has a final object given by $x$. Furthermore, by  distributivity,  we also see that  $\udl{\posetExample}_x\subseteq \udl{\posetExample}$ is closed under nontrivial finite (co)products, and it is immediate to check that $\udl{\posetExample}_x$ is again distributive. Hence, $\udl{\posetExample}_x$ is another complementable parametrised poset. Moreover, it is also immediate to see that $\udl{\posetExample}_x\subseteq \udl{\posetExample}$ is downward--closed.
\end{cons}

\begin{lem}[Complement decompositions]\label{lem:complement_decompositions}
    Let $\udl{\posetExample}$ be a complementable parametrised poset and $x\in \udl{\posetExample}$.  Then the functors $(x\wedge-)\times (x^c\wedge-)\colon \udl{\posetExample}\rightarrow \udl{\posetExample}_{x}\times \udl{\posetExample}_{x^c}$ and  $-\vee-\colon \udl{\posetExample}_{x}\times \udl{\posetExample}_{x^c}\rightarrow \udl{\posetExample}$ are inverse equivalences to each other.
\end{lem}
\begin{proof}
    This is an immediate consequence of the fact that for any $a, b, c\in \udl{\posetExample}$, the natural maps 
    \[(x\wedge a)\vee (x^c\wedge a) \xlongrightarrow{\simeq} (x\vee x^c)\wedge a \simeq \distributiveFinalObject\wedge a\simeq a\]
    \[x\wedge b \xlongrightarrow{\simeq} x\wedge (x \wedge b) \xlongrightarrow{\simeq} x\wedge ((x\wedge b)\vee(x^c\wedge c))\]
    \[x^c\wedge c \xlongrightarrow{\simeq} x^c\wedge (x^c \wedge c) \xlongrightarrow{\simeq} x^c\wedge ((x\wedge b)\vee(x^c\wedge c))\]
    are equivalences.
\end{proof}

\begin{cor}[Smashing (co)localisations]\label{cor:smashing_colocalisations}
    Let $\udl{\posetExample}$ be a complementable parametrised poset and $x\in \udl{\posetExample}$. Then we have the Bousfield (co)localisations
    \begin{center}
        \begin{tikzcd}
            \udl{\posetExample }\ar[rr, "(x\wedge -)"description, two heads] && \udl{\posetExample}_x.\ar[ll, bend right =40, hook, "(\distributiveInitialObject\vee-)"'description]\ar[ll, bend right =-40, hook, "(x^c\vee-)"description]
        \end{tikzcd}
    \end{center}
\end{cor}
\begin{proof}
    Since $\udl{\posetExample}_{x^c}$ has an initial and final object by \cref{cons:smashing_subposet}, we have a Bousfield (co)localisation 
    \begin{center}
        \begin{tikzcd}
            \udl{\posetExample}_{x^c}\ar[rr, two heads] && \terminalTCat.\ar[ll, bend right =30, hook, "\distributiveInitialObject"']\ar[ll, bend right =-30, hook, "x^c"]
        \end{tikzcd}
    \end{center}
    Thus by applying $\udl{\posetExample}_x\times-$ to this Bousfield (co)localisation and using the decomposition from \cref{lem:complement_decompositions}, we obtain the desired adjunctions.
\end{proof}

\begin{lem}\label{lem:intersection_left_adjoint}
    Suppose $\udl{\posetExample}$ is a complementable parametrised poset and $x\in \udl{\posetExample}$. Then the functor $x\vee-\colon \udl{\posetExample}\rightarrow \udl{\posetExample}$ admits a left adjoint given by $x^c\wedge -$.
\end{lem}
\begin{proof}
    Note that the functor $x\vee-\colon \udl{\posetExample}\rightarrow\udl{\posetExample}$ is identified with the composite $\udl{\posetExample}\xrightarrow{x^c\wedge-} \udl{\posetExample}_{x^c}\xrightarrow{x\vee-} \udl{\posetExample}$. Thus, by \cref{cor:smashing_colocalisations}, this admits a left adjoint given by $\udl{\posetExample}\xrightarrow{x^c\wedge-} \udl{\posetExample}_{x^c}\xrightarrow{\distributiveInitialObject\vee-} \udl{\posetExample}$, which is identified with $x^c\wedge-$.
\end{proof}

We are now ready to define and elaborate on the properties of an excisable poset, which is supposed to axiomatise the notion of the star of an $(n+1)$--cube (i.e. the subposet of $(\Delta^1)^{\times n}$, thought of as the poset of subsets of the set with $n+1$ elements, consisting of those subsets of size at most 1) that appears in Goodwillie's definition of strongly cocartesian cubes.

\begin{defn}\label{defn:excisable_poset}
    An \textit{excisable poset} is the datum of a complementable poset $\udl{\posetExample}\in\posetCategory_{\baseCat}$ together with a nonempty full subposet $\sigma\colon \subPoset{\udl{\posetExample}}\subseteq \udl{\posetExample}$ which is downward closed, i.e. for all $u\in \baseCat$, if $x\rightarrow y$ is a morphism in $\posetExample(u)$ where $y\in \subPoset{\posetExample}(u)$, then $x$ is also in $\subPoset{\posetExample}(u)$.
\end{defn}

\begin{obs}
    Note in particular that $\subPoset{\udl{\posetExample}}$ also contains an initial object.
\end{obs}

\begin{lem}\label{lem:smashing_excisable_poset}
    Let $\udl{\posetExample}$ be a complementable parametrised poset, $x\in\udl{\posetExample}$, and $\subPoset{\posetExample}\subseteq \udl{\posetExample}$ a downward--closed subposet. Then writing $\subPoset{\udl{\posetExample}}_x\subseteq \udl{\posetExample}_x$ for the image of the composite $\subPoset{\udl{\posetExample}}\subseteq \udl{\posetExample}\xrightarrow{(x\wedge-)}\udl{\posetExample}_x$, we have that $\subPoset{\udl{\posetExample}}_x\subseteq \udl{\posetExample}_x$ is also downward--closed. 
\end{lem}
\begin{proof}
    Let $u\in \subPoset{\udl{\posetExample}}_x$, so that $u=x\wedge m$ for some $m\in\subPoset{\udl{\posetExample}}$, and suppose we have a morphism $v\rightarrow u$ in $\udl{\posetExample}_x$. Thus, we have the morphism $v\rightarrow u=x\wedge m\rightarrow m$ in $\udl{\posetExample}$ with $m\in\subPoset{\udl{\posetExample}}$. Hence, since $\subPoset{\udl{\posetExample}}\subseteq\udl{\posetExample}$ is downward--closed, we also have that $v\in \subPoset{\udl{\posetExample}}$. Furthermore, since we have the morphism $x^c\wedge v\rightarrow x^c\wedge x\wedge m = \distributiveInitialObject$ in $\udl{\posetExample}$, we see  that $x^c\wedge v=\distributiveInitialObject$ by \cref{rmk:smashing_with_initial_object}, and so $x\wedge v\rightarrow v$ is an equivalence. All in all, we see that $v\in \subPoset{\udl{\posetExample}}_x$ as claimed.
\end{proof}

To end this subsection, we introduce the analogues of strong cocartesianness and cartesianness in our setup following \cite{stoll}, and establish some technical properties needed in the proof of \cref{alphthm:excisive_approximation_left_adjoint}.

\begin{nota}
    Given a complementable parametrised poset $\udl{\posetExample}$, we let $\mathring{\sigma}\colon\udl{\puncture{\posetExample}}\subseteq \udl{\posetExample}$ be the full subposet away from the initial object. Moreover, given an excisable poset $\sigma$, we write $\initialObject{\sigma}\colon\terminalTCat\hookrightarrow \subPoset{\udl{\posetExample}}$ for the inclusion of the initial object.
\end{nota}

\begin{terminology}
    Let $f\colon \udl{I}\rightarrow \udl{J}$ be a functor and $\udl{\sC}\in\cat_{\basecat}$. We say that $\udl{\sC}$ is \textit{$f$--left--extensible (resp. $f$--right--extensible)} if the left Kan extensions (resp. right Kan extensions) along $f$ exist in $\udl{\sC}$.
\end{terminology}

\begin{defn}\label{defn:sigma_cocartesian}
    An object in $\udl{\func}(\udl{\posetExample},\udl{\sC})$ is said to be \textit{$\sigma$--cocartesian} if it is in the image of the fully faithful functor $\sigma_!\colon \udl{\func}(\subPoset{\udl{\posetExample}},\udl{\sC})\hookrightarrow \udl{\func}({\udl{\posetExample}},\udl{\sC})$. It is said to be \textit{cartesian} if it is in the image of the fully faithful functor $\puncture{\sigma}_*\colon \udl{\func}(\udl{\puncture{\posetExample}},\udl{\sC})\hookrightarrow \udl{\func}({\udl{\posetExample}},\udl{\sC})$. We write $\cocartesianFunc{\sigma}(\udl{\posetExample},\udl{\sC})$ and $\cartesianFunc(\udl{\posetExample},\udl{\sC})$ for the full subcategories of $\sigma$--cocartesian and cartesian functors, respectively.
\end{defn}

\begin{prop}\label{prop:squashing_preserves_cocartesianness}
    Let $\sigma$ be an excisable poset, $x\in \udl{\posetExample}$, and $D\in \cocartesianFunc{\sigma}(\udl{\posetExample},\udl{\sC})$. Then the functor $D\circ (x\vee -)\colon \udl{\posetExample}\rightarrow \udl{\sC}$ is also $\sigma$--cocartesian.
\end{prop}
\begin{proof}
    We need to show that the adjunction counit $\sigma_!\sigma^*(x\vee-)^*D\rightarrow (x\vee-)^*D$ is an equivalence. Recall that we had an adjunction $x^c\wedge-\colon \udl{\posetExample}\rightleftharpoons \udl{\posetExample} : x\vee-$ from \cref{lem:intersection_left_adjoint}, and so in particular, we have $(x\vee-)^*\simeq (x^c\wedge-)_!$. Next, we notice that there is a dashed factorisation in the diagram
    \begin{center}
        \begin{tikzcd}
            \udl{\posetExample} \rar["x^c\wedge-"] & \udl{\posetExample}\\
            \subPoset{\udl{\posetExample}} \uar["\sigma", hook]\rar[dashed, "x^c\wedge-"] & \subPoset{\udl{\posetExample}}.\uar["\sigma", hook]
        \end{tikzcd}
    \end{center}
    This is because for $a\in\subPoset{\udl{\posetExample}}$, the adjunction counit $x^c\wedge(x\vee a)\rightarrow a$ may be precomposed with $x^c\wedge a\rightarrow x^c\wedge(x\vee a)$ to give a map $x^c\wedge a\rightarrow a$. Hence, since $\subPoset{\udl{\posetExample}}$ was downward--closed, $x^c\wedge a$ is also in $\subPoset{\udl{\posetExample}}$. Since $D$ was $\sigma$--cocartesian, we may write $D=\sigma_!\overline{D}$. Now, we consider the computation:
    \begin{equation*}
        \begin{split}
            \sigma_!\sigma^*(x\vee-)^*D &\simeq \sigma_!\sigma^*(x\vee-)^*\sigma_!\overline{D}\\
            & \simeq \sigma_!\sigma^*(x^c\wedge-)_!\sigma_!\overline{D} \\
            &\simeq \sigma_!\sigma^*\sigma_!(x^c\wedge-)_!\overline{D} \\
            &\simeq \sigma_!(x^c\wedge-)_!\overline{D} \\
            &\simeq (x^c\wedge-)_!\sigma_!\overline{D} \\
            &\simeq (x\vee-)^*\sigma_!\overline{D} = (x\vee-)^*D
        \end{split}
    \end{equation*}
    where the fourth equivalence used that $\sigma^*\sigma_! \simeq \id$ by virtue of fully faithfulness.
\end{proof}

We now work towards the cartesian analogue of the result above.

\begin{lem}\label{lem:collapse_Cartesianness}
    Let $\udl{P}, \udl{I}, \udl{\sC}\in\cat_{\baseCat}$ where $\udl{P}$ and $\udl{I}$ have initial objects,  $\udl{I}$ has a final object not equivalent to the initial object, and $\udl{\sC}$ has the appropriate Kan extensions. Then for any $F\colon \udl{P}\rightarrow \udl{\sC}$, the diagram $F\circ \pi\colon \udl{P}\times \udl{I}\xrightarrow{\pi} \udl{P}\xrightarrow{F}\udl{\sC}$ is cartesian. 
\end{lem}
\begin{proof}
    Since $\underline{I}$ has a final object, the map $\pi$ has a right adjoint $r$. Moreover because the final object is different then initial object, the projection $\overline{\pi}\colon (\udl{P}\times \udl {I})\backslash \varnothing \rightarrow \udl{P}$ also has a right adjoint $\overline{r}$. It is easy to see that we have the solid commuting diagram
    \begin{center}
        \begin{tikzcd}
            \underline{P} \dar["\overline{r}"]\rar["r",shift right = 1]&\underline{P}\times\underline{I}\lar[bend right = 40, dashed , "\pi"']\\
            (\underline{P}\times\underline{I})\backslash\varnothing\ar[ur, "i"',hook]\ar[u, bend left = 50, "\overline{\pi}",dashed]
        \end{tikzcd}
    \end{center}
    Note that $r^*\dashv\pi^*$ and $\overline{r}^*\dashv \overline{\pi}^*$. Therefore we compute that 
    \[\pi^* F \simeq r_* F \simeq i_* \overline{r}_* F \simeq i_* \overline{\pi}^* F,\] proving the claim.
    % In particular, this means that $r^*\dashv\pi^*$ and $\overline{r}^*\dashv \overline{\pi}^*$. Therefore, by considering the left adjoint diagram which commutes by the above observations, we see that the triangle 
    % \begin{center}
    %    \begin{tikzcd}
    %        \underline{\func}(\underline{P},\underline{\sC}) \dar["\overline{\pi}^*"']\rar["\pi^*"]&\underline{\func}(\underline{P}\times\underline{I},\underline{\sC})\\
    %       \underline{\func}((\underline{P}\times\underline{I})\backslash\distributiveInitialObject,\underline{\sC})\ar[ur, "i_*"',hook]
    %    \end{tikzcd}
    % \end{center}
    % also commutes, whence the claim.
\end{proof}

\begin{prop}\label{prop:squashing_induces_cartesianness}
    Let $\udl{\sC}\in\cat_{\baseCat}$, $\sigma$ an excisable poset, and $x\neq\distributiveInitialObject\in \udl{\posetExample}$. Then for any $D\colon \udl{\posetExample}\rightarrow\udl{\sC}$, the diagram $D\circ (x\vee-)$ is cartesian.
\end{prop}
\begin{proof}
    Using \cref{lem:complement_decompositions}, we see that the functor $(x\vee-)\colon \udl{\posetExample}\rightarrow\udl{\posetExample}$ factors as
    \[(x\vee-)\colon \udl{\posetExample}\simeq \udl{\posetExample}_{x^c}\times \udl{\posetExample}_x \xlongrightarrow{\pi} \udl{\posetExample}_{x^c} \xlongrightarrow{(x\vee-)}\udl{\posetExample}\]
    Since $x\neq \distributiveInitialObject$, $\udl{\posetExample}_x$ has initial and final objects which are not equivalent. Thus, we may apply \cref{lem:collapse_Cartesianness} to see that postcomposing with $F\colon \udl{\posetExample}\rightarrow \udl{\sC}$ yields a cartesian diagram, as wanted.
\end{proof}

\subsection{Face decompositions}\label{section:basicCubeYoga}\label{subsec:face_decompositions}

We now axiomatise the important manoeuvre of decomposing cubes into its faces. We first work towards formulating and proving a generalisation of the classical statement (cf. for instance \cite[Prop. 6.1.1.13]{lurieHA}) that if a cube is strongly cocartesian, then so are all its faces. The materials here are needed in order to relate different notions of excisiveness later across different excisable posets, but are not needed for the proof of \cref{alphthm:excisive_approximation_left_adjoint}. Since the proofs are rather technical,  the reader should feel free to skip this subsection on first reading.

To define the faces of a complementable parametrised poset we require the following definition.

\begin{defn}
Given a complementable parametrised poset, we say $a,d\in \udl{L}$ are disjoint if $a\wedge d = \distributiveInitialObject$. We say three elements $a,d,z\in \udl{\posetExample}$ are \textit{a decomposition of $\udl{\posetExample}$} if $a,d$ and $z$ are pairwise disjoint and $a\vee d\vee z = \distributiveFinalObject$.
\end{defn}

\begin{rmk}
Note that given two disjoint elements $a,d \in \udl{\posetExample}$, there is a unique $z\in \udl{\posetExample}$ such that $a,d,z$ is a decomposition of $\udl{L}$. In fact, $z$ must be equal to $(a\vee d)^c$.
\end{rmk}

\begin{cons}[Face posets along triple decompositions]\label{cons:face_posets}\label{cons:intersectingSingletons}
    Let $\udl{\posetExample}$ be a complementable parametrised poset and let $a, d, z\in\udl{\posetExample}$ be a decomposition of $\udl{\posetExample}$. Thus recalling the notations from \cref{cons:smashing_subposet} and the Bousfield (co)localisation from \cref{cor:smashing_colocalisations}, we obtain the Bousfield (co)localisations
    \begin{center}
        \begin{tikzcd}
            \udl{\posetExample}\ar[rrrr, "\cap_{z}\coloneqq((a\vee d)\wedge -)"description, two heads] &&&& \udl{\posetExample}_{a\vee d}\ar[llll, bend right =40, hook, "A_z\coloneqq (\distributiveInitialObject\vee-)"'description]\ar[llll, bend right =-40, hook, "\Omega_z\coloneqq (z\vee-)"description]&&\udl{\posetExample}_{a\vee d}\ar[rrrr, "\cap\coloneqq (a\wedge -)"description, two heads] &&&& \udl{\posetExample}_a.\ar[llll, bend right =40, hook, "A\coloneqq (\distributiveInitialObject\vee-)"'description]\ar[llll, bend right =-40, hook, "\Omega\coloneqq (d\vee-)"description] 
        \end{tikzcd}
    \end{center}
    The \textit{$a$--face associated to the decomposition $a,d,z$} is defined to be the composite
    \begin{equation}\label{eqn:face_subposet}
        \begin{tikzcd}
            \Phi\coloneqq \Phi_d\coloneqq \Phi_{a,d,z}\colon \udl{\posetExample}_a \rar[hook, "\Omega"] & \udl{\posetExample}_{a\vee d} \rar[hook, "A_z"] & \udl{\posetExample}
        \end{tikzcd}
    \end{equation}
    which is a fully faithful functor. This generalises \cite[Def. 6.1.1.12]{lurieHA} to our setup.

    \end{cons}

    Futhermore, if we have an excisable structure on the original poset, then there is also an induced excisable structure on the face posets as the following point explains.

\begin{cons}[Excisable structure on face posets]\label{cons:excisable_structure_on_face_posets}
      Suppose we have the structure of an excisable poset $\sigma\colon \subPoset{\udl{\posetExample}}\subseteq \udl{\posetExample}$ and suppose $a,d,z$ are a decomposition of $\udl{\posetExample}$. Then by \cref{lem:smashing_excisable_poset,cons:smashing_subposet}, we also obtain the structure of an induced excisable poset $\sigma_a\colon \subPoset{\udl{\posetExample}}_a\subseteq \udl{\posetExample}$ on $\udl{\posetExample}_a$ by taking $\subPoset{\udl{\posetExample}}_a$ to be the image of $\subPoset{\udl{\posetExample}}\subseteq \udl{\posetExample}\xrightarrow{\cap\circ \cap_z}\udl{\posetExample}_a$. Now consider the pullbacks
    \begin{equation}\label{eqn:pullback_intersecting_subposets}
    \begin{tikzcd}
        \udl{\I} \rar["j",hook]\dar["\overline{\cap}"', shift right = 1, two heads] \ar[dr, phantom , "\lrcorner"] & \udl{\posetExample}_{a\vee d}\dar["\cap"', two heads] && \udl{\J}\rar[hook, "\tau"] \dar[hook, "i"]\ar[dr, phantom , "\lrcorner"]& \udl{\posetExample}_{a\vee d} \dar["A_z", hook]\\
        \subPoset{\udl{\posetExample}}_a \rar[hook, "\sigma_a"']\ar[u, bend right = 40, shift left = 1,dashed, "\overline{\Omega}"' yshift = -6] & \udl{\posetExample}_a. \ar[u, bend right = 40, dashed, "\Omega"'] && \subPoset{\udl{\posetExample}} \rar["\sigma",hook] & \udl{\posetExample}
    \end{tikzcd}
\end{equation}
For the left square in \cref{eqn:pullback_intersecting_subposets}, since the  arrow $\cap$ admits a fully faithful right adjoint, by standard abstract nonsense, so does the functor $\overline{\cap}$, which we call $\overline{\Omega}$; by construction, this also satisfies that the canonical map $j\circ \overline{\Omega}\rightarrow \Omega\circ \sigma_a$ is an equivalence. For the right square, since both $\sigma\colon \subPoset{\udl{\posetExample}}\subseteq \udl{\posetExample}$ and $A_z\colon \udl{\posetExample}_{a\vee d}\subseteq \udl{\posetExample}$ are downward--closed, so is the inclusion $\tau\colon \udl{\J}\subseteq \udl{\posetExample}_{a\vee d}$, and so it provides an excisable poset structure on $\udl{\posetExample}_{a\vee d}$. 

Moreover, since by definition $\subPoset{\udl{\posetExample}}\subseteq \udl{\posetExample}\xrightarrow{\cap\circ \cap_z}\udl{\posetExample}_a$ lands in $\subPoset{\udl{\posetExample}}_a$, we get a  square
\begin{center}
    \begin{tikzcd}
        \udl{\J} \rar[hook, "\tau"]\dar[ shift right = 0, two heads] & \udl{\posetExample}_{a\vee d}\dar["\cap"', two heads]\\
        \subPoset{\udl{\posetExample}}_a \rar[hook, "\sigma_a"'] & \udl{\posetExample}_a, 
    \end{tikzcd}
\end{center}
whence $\tau\colon\udl{\J}\subseteq \udl{\posetExample}_{a\vee d}$ factorises via an inclusion $k\colon \udl{\J}\hookrightarrow \udl{\I}$.
\end{cons}

\begin{prop}\label{prop:faces_of_strong_cocartesian}
    Let $\udl{\sC}$ be a parametrised cocomplete category,  $\sigma \colon\subPoset{\udl{\posetExample}}\subseteq \udl{\posetExample}$ be an excisable poset, and $a\in \udl{\posetExample}$. Let $X\colon \udl{\posetExample}\rightarrow\udl{\sC}$ be a $\sigma$--cocartesian diagram. Then every $a$--face of $X$ is $\sigma_a$--cocartesian.
\end{prop}
\begin{proof}
We borrow the notation from \cref{cons:excisable_structure_on_face_posets}. We fix the $a$-face $\Phi\colon \udl{\posetExample}_a\hookrightarrow\udl{\posetExample}$ from \cref{eqn:face_subposet} associated to the decomposition $a,d,z$ of $\udl{\posetExample}$. Our goal is to show that the functor $\Phi^*X\in \udl{\func}(\udl{\posetExample}_a,\udl{\sC})$ is $\sigma_a$-cocartesian, in other words that the map $\sigma_{a!}\sigma_a^*\Phi^*X\rightarrow\Phi^*X$ is an equivalence.

First consider the right square in \cref{eqn:pullback_intersecting_subposets}. A simple computation using the formula for pointwise Kan-extensions and the fact that $\udl{\posetExample}_{a\vee d}$ is downward closed in $\udl{\posetExample}$ shows that $Y\coloneqq A_z^*X\in\udl{\func}(\udl{\posetExample}_{a\vee d},\udl{\sC})$ is $\tau$--cocartesian, i.e. the map $\tau_!\tau^*Y\rightarrow Y$ is an equivalence. Now consider the diagram
\begin{center}
    \begin{tikzcd}
    \udl{\posetExample}_a\ar[rr, bend left = 30, "\Phi^*X"] \rar[shift right = 1,"\Omega"'] & \udl{\posetExample}_{a\vee d} \rar["Y"] \lar[shift right =1, "{\cap}"']& \udl{\sC} \\
    \subPoset{\udl{\posetExample}}_a\uar[hook, "\sigma_a"] \rar["\overline{\Omega}"', shift right = 1] & \udl{\I} \uar["j"', hook]\lar[shift right = 1, "\overline{\cap}"']\\
    & \udl{\J}\uar["k"', hook]\ar[uu, bend right = 40, "\tau"', hook]
    \end{tikzcd}
\end{center}
coming from \cref{cons:excisable_structure_on_face_posets}, where the squares commute. The adjunctions  induce adjunctions $\Omega^*\dashv \cap^*$ and $\overline{\Omega}^*\dashv \overline{\cap}^*$  on the respective functor categories valued in $\udl{\sC}$, and so in particular we get $\sigma_{a!}\overline{\Omega}^*\simeq \sigma_{a!}\overline{\cap}_!\simeq \cap_!j_!\simeq {\Omega}^*j_!$. Moreover, we also have 
$j^*Y\simeq j^*\tau_!\tau^*Y\simeq j^*j_!k_!k^*j^*Y\simeq k_!k^*j^*Y$, using that $j$ is fully faithful and that $Y$ is $\tau$-cocartesian. Thus, putting everything together, we obtain an equivalence
\[\Phi^*X\simeq \Omega^*Y\simeq \Omega^*j_!j^*Y\simeq \sigma_{a!}\overline{\Omega}^*j^*Y\simeq \sigma_{a!}\sigma_a^*\Omega^*Y\simeq \sigma_{a!}\sigma_a^*\Phi^*X,\] as was to be shown.
\end{proof}

Next, we shall prove a generalisation of the fact that if all the faces of a cube are cartesian, then so is the cube itself. We will require the following lemma.

\begin{lem}\label{lem:adjoint_of_unioning_with_a_subset}
    Let $\udl{\posetExample}$ be a complementable parametrised poset and let $a,d,z\in \udl{\posetExample}$ be a decomposition of $\udl{\posetExample}$. Then the face functor $\Phi=\Phi_d$ from \cref{eqn:face_subposet} induces a Bousfield colocalisation
    \begin{center}
        \begin{tikzcd}
            \Phi'_d = d\vee(-) \colon \udl{\posetExample}_a \ar[rr, hook, shift left = 1] && \udl{\posetExample}_{d/}  \ar[ll, shift left = 1] \cocolon \cap' \coloneqq a\wedge(-).
        \end{tikzcd}
    \end{center}
\end{lem}
\begin{proof}
    We need to show that $\myuline{\map}_{\udl{\posetExample}_{d/}}(d\vee x, t) \simeq \myuline{\map}_{\udl{\posetExample}_a}(x, a\wedge t)$ for all $x\in \udl{\posetExample}_a$ and $t\in \udl{\posetExample}_{d/}$. Note that there is a map $d\rightarrow t$ and that $a\wedge x\rightarrow x$ is an equivalence. Therefore we can compute that both sides (which are equivalent apriori to either $\terminalTCat$ or $\udl{\distributiveInitialObject}$) are nonempty if and only if there is a map $x\rightarrow t$ in $\udl{\posetExample}$, whence the desired equivalence. 
\end{proof}

\begin{prop}\label{prop:face_of_cartesian}
    Let $\udl{\sC}$ be a parametrised cocomplete category, $\udl{\posetExample}$ be a complementable parametrised poset, and $a\in \udl{\posetExample}$. Let $X\colon \udl{\posetExample}\rightarrow\udl{\sC}$ be a diagram whose $a$--faces are all cartesian. Then $X$ is itself cartesian.
\end{prop}
\begin{proof}
    Recalling the notation $\mathring{\sigma}\colon \mathring{\udl{\posetExample}}\subseteq \udl{\posetExample}$ for the inclusion away from the initial object, we need to show that the canonical map $X\rightarrow \mathring{\sigma}_*\mathring{\sigma}^*X$ is an equivalence. We break up the problem into two steps, and for this, consider $\udl{\M}\subseteq \mathring{\udl{\posetExample}}$ the subposet of all objects that intersect nontrivially with $a$, i.e. the pullback
    \begin{center}
    \begin{tikzcd}
        \underline{\M} \rar["m",hook]\dar \ar[dr, phantom, very near start, "\lrcorner"] & \mathring{\udl{\posetExample}}\dar["\cap\circ \cap_z\circ \mathring{\sigma}"]\\
        \mathring{\udl{\posetExample}}_a \rar["\mathring{\sigma}_a",hook] & \udl{\posetExample}_a.
    \end{tikzcd}
\end{center}
    
    Our first claim is that $\mathring{\sigma}^*X$ is a right Kan extension of $m^*\mathring{\sigma}^*X$, i.e. the map $\mathring{\sigma}^*X\rightarrow m_*m^*\mathring{\sigma}^*X$ is an equivalence. This may be checked for every parametrised point in $\mathring{\udl{\posetExample}}$. By passing to the appropriate restrictions if necessary, without loss of generality, suppose we have a point $d\in\mathring{\udl{\posetExample}}\backslash\udl{\M}$ so that $a\wedge d=\distributiveInitialObject$. Writing $z\coloneqq (a\vee d)^c$, we thus get the triple decomposition $\distributiveFinalObject=a\vee d\vee z$ and along with it, an $a$--face $\Phi_d=\Phi_{a,d,z}\colon \udl{\posetExample}_a\hookrightarrow \udl{\posetExample}$ as in \cref{eqn:face_subposet} which sends $\distributiveInitialObject$ to $d\in\mathring{\udl{\posetExample}}$ (in particular lands in $\mathring{\udl{\posetExample}}\subseteq \udl{\posetExample}$). Writing $\mathring{\Phi}_d\colon \udl{\posetExample}_a\hookrightarrow \udl{\mathring{\posetExample}}$ for this induced map, we will thus show that 
    \[\mathring{\Phi}_d^*\mathring{\sigma}^*X\longrightarrow \mathring{\Phi}_d^*m_*m^*\mathring{\sigma}^*X\]
    is an equivalence. To see this, consider the commuting diagram
    \begin{center}
    \begin{tikzcd}
    \underline{\M} \ar[rr,"m", hook]&& \mathring{\udl{\posetExample}}\ar[dr, "\mathring{\sigma}", hook]\\
    \underline{\M}_{d/} \ar[urr, phantom,  "\rotatebox{90}{\scalebox{1.5}{$\lrcorner$}}"]\ar[rr,"\overline{m}", hook]\dar["{\overline{\nu}}", dashed, shift left = 1] \uar["\overline{\ell}", hook]&& \mathring{\udl{\posetExample}}_{d/} \uar[hook, "\ell"] \rar["\mathring{\overline{\sigma}}", hook] \dar["{\nu}", dashed, shift left = 1]& \udl{\posetExample} \rar["X"] & \underline{\sC}\\
    \mathring{\udl{\posetExample}_a}\ar[rr,"\mathring{\sigma}_a"', hook] \uar["\overline{\varphi}", shift left = 2, hook] && \udl{\posetExample}_a \uar["\varphi", shift left = 2, hook] \ar[ur, "\Phi_d"'] 
    \end{tikzcd}
\end{center}
    where all squares commute including the one involving the dashed maps, by construction of $\udl{\M}$, and the adjunctions are by \cref{lem:adjoint_of_unioning_with_a_subset}. Observe that since the top square is a pullback involving fully faithful functors between posets, another computation using the formula for pointwise Kan extensions yields that the Beck--Chevalley map $\ell^*m_*\rightarrow \overline{m}_*\overline{\ell}^*$ is an equivalence. Furthermore, we also have that $\overline{\nu}^*\dashv \overline{\varphi}^*$ and $\nu^*\dashv \varphi^*$, so that $\overline{\varphi}^*\simeq \overline{\nu}_*$ and ${\varphi}^*\simeq {\nu}_*$. Therefore, we may now compute:
    \begin{equation*}
    \begin{split}
    \mathring{\Phi}_d^*m_*m^*\mathring{\sigma}^*X &\simeq \varphi^*\ell^*m_*m^*\mathring{\sigma}^*X\\
    &\simeq {\varphi}^*\overline{m}_{*}\overline{\ell}^*m^*\mathring{\sigma}^*X\\
        & \simeq{\varphi}^*\overline{m}_{*}\overline{m}^*\mathring{\overline{\sigma}}^*X \\ 
        & \simeq{\nu}_{*}\overline{m}_{*}\overline{m}^*\mathring{\overline{\sigma}}^*X\\
        &\simeq \mathring{\sigma}_{a*}\overline{{\nu}}_{*}\overline{m}^*\mathring{\overline{\sigma}}^*X \\
        &\simeq \mathring{\sigma}_{a*}\overline{\varphi}^*\overline{m}^*\mathring{\overline{\sigma}}^*X\\
        &\simeq \mathring{\sigma}_{a*}\mathring{\sigma}_{a}^*\Phi_d^*X\\
         &\simeq \Phi_d^*X\\
         &\simeq \mathring{\Phi}_d^*\mathring{\sigma}^*X
    \end{split}
\end{equation*}
    as desired, where the penultimate equivalence is because the $a$--face $\Phi_d^*X$ is cartesian.

    For the second step, recall that ultimately we want to show that the canonical map $X \rightarrow \mathring{\sigma}_*\mathring{\sigma}^*X$ is an equivalence. But since the inclusion $\mathring{\sigma} \colon \mathring{\udl{\posetExample}}\hookrightarrow \udl{\posetExample}$ only adds the initial object $\distributiveInitialObject$, it suffices to show that $\Phi_{\distributiveInitialObject}^*X \rightarrow \Phi_{\distributiveInitialObject}^*\mathring{\sigma}_*\mathring{\sigma}^*X$ is an equivalence. In this case, we would need a similar diagram as in the previous step adjusted by the fact that the case $d =\distributiveInitialObject$ is special and slightly simpler:
\begin{center}
    \begin{tikzcd}
    \underline{\M} \rar["m",hook] \dar["\overline{{\nu}}", dashed, shift left = 2]& \mathring{\udl{\posetExample}} \rar["\mathring{\sigma}", hook] & \udl{\posetExample} \rar["X"] \dar["{\nu}", dashed, shift left = 2]& \underline{\sC}\\
    \mathring{\udl{\posetExample}_a} \ar[rr,"\mathring{\sigma}_a"', hook] \uar["\overline{\Phi}_{\distributiveInitialObject}", shift left = 2, hook] && \udl{\posetExample}_a.\uar["{\Phi}_{\distributiveInitialObject}", shift left = 2, hook] 
    \end{tikzcd}
\end{center}
Similarly as in the previous step, we may now compute:
\begin{equation*}
    \begin{split}
    \Phi_{\distributiveInitialObject}^*\mathring{\sigma}_*\mathring{\sigma}^*X &\simeq \Phi_{\distributiveInitialObject}^*\mathring{\sigma}_*m_*m^*\mathring{\sigma}^*X\\
    &\simeq \nu_{*}\mathring{\sigma}_*m_*m^*\mathring{\sigma}^*X\\
    &\simeq \mathring{\sigma}_{a*}\overline{{\nu}}_{*}m^*\mathring{\sigma}^*X\\
    &\simeq \mathring{\sigma}_{a*}\overline{{\Phi}}_{\distributiveInitialObject}^*m^*\mathring{\sigma}^*X\\
    &\simeq \mathring{\sigma}_{a*}\mathring{\sigma}_a^*\Phi_{\distributiveInitialObject}^*X\\
    &\simeq \Phi_{\distributiveInitialObject}^*X
    \end{split}
\end{equation*}
where the first equivalence is by the first step above, and the last equivalence is since the $a$--face $\Phi_{\distributiveInitialObject}^*X$ is cartesian. This completes the proof of the proposition.
\end{proof}

\subsection{Excision}\label{subsec:excision}

Armed with the requisite theory of posets, we are ready to formulate what excisiveness means in our context and work towards proving \cref{alphthm:excisive_approximation_left_adjoint} on the formula for excisive approximations. As already alluded to in the introduction, as in ordinary Goodwillie calculus, under presentability hypotheses, say, the mere existence of this approximation is always true. The real value of the theorem, which necessitates the work in this subsection, is that this approximation has an explicit formula given by \cref{cons:excisive_approximation_functors}. This will play a crucial role in our proof of \cref{mainthm:spherically_faithful_implies_semiadditivity} in \cref{section:parametrised_cubical_excision}.

\begin{defn}\label{defn:excisive_functors}
    Let $F\colon \udl{\sC}\rightarrow \udl{\D}$ be a functor where $\udl{\sC}$ is $\sigma$--left--extensible and $\udl{\D}$ is $\puncture{\sigma}$--right--extensible. We say that $F$ is \textit{$\sigma$--excisive} if the diagram in $\cat_{\baseCat}$
    \begin{center}
        \begin{tikzcd}
            \udl{\func}(\udl{\posetExample},\udl{\sC})\rar["F"] & \udl{\func}(\udl{\posetExample},\udl{\D})\\
            \udl{\func}(\subPoset{\udl{\posetExample}},\udl{\sC}) \uar[hook, "\sigma_!"]\ar[r,dashed]& \udl{\func}(\udl{\puncture{\posetExample}},\udl{\sC})\uar[hook, "\puncture{\sigma}_*"']
        \end{tikzcd}
    \end{center}
    admits the dashed factorisation. In other words, it sends $\sigma$--cocartesian diagrams to cartesian ones. We denote by $\udl{\excisiveCat}^{\sigma}(\udl{\sC},\udl{\D})$ for the full subcategory of $\udl{\func}(\udl{\sC},\udl{\D})$ on the $\sigma$--excisive functors.
\end{defn}

\begin{rmk}
    Since $\puncture{\sigma}_*\colon \udl{\func}(\udl{\puncture{\posetExample}},\udl{\sC})\rightarrow \udl{\func}(\udl{\posetExample},\udl{\D})$ is fully faithful, it is really just a property that the dashed factorisation exists.
\end{rmk}

\begin{rmk}
    Note that in \cref{defn:excisive_functors} we have required the factorisation to happen in $\cat_{\baseCat}$. Thus, this notion is stable under basechange, i.e. for all $u\in\baseCat$, restriction to  $\cat_{\baseCat_{/u}}$ still yields a factorisation of $\baseCat_{/u}$--categories.
\end{rmk}

In order to state the theorem, we will need one last piece of terminology.

\begin{defn}
    Let $\sigma$ be an excisable poset. A category $\udl{\D}\in\cat_{\baseCat}$ is said to be \textit{$\sigma$--differentiable} if it admits sequential colimits and $\puncture{\udl{\posetExample}}$--indexed limits, and that these commute with each other.
\end{defn}

We now begin to work on proving our first main result in earnest. We first collect some auxiliary constructions and their properties that will go into \cref{cons:excisive_approximation_functors}. The next construction should of course be compared to \cite[Cons. 6.1.1.22]{lurieHA}.

\begin{cons}[Pre--excisive approximation functor]\label{cons:pre-excisive_approximation_functors}
    Suppose $\udl{\sC}$ has a final object. We construct a functor
    \begin{equation}\label{eqn:goodwillie_T_functor}
        \goodwillieTFunctor_{\sigma}\colon \udl{\func}(\udl{\sC},\udl{\D})\longrightarrow \udl{\func}(\udl{\sC},\udl{\D})
    \end{equation}
    To this end, first consider the functor $\padding_{\sigma}$ given by the composite of fully faithful functors
    \[\padding_{\sigma}\colon \udl{\sC}\xhookrightarrow{\initialObject{\sigma}_*} \udl{\func}(\subPoset{\udl{\posetExample}},\udl{\sC}) \xhookrightarrow[\simeq]{\sigma_!}\cocartesianFunc{\sigma}(\udl{\posetExample},\udl{\sC}).\]
    The functor in \cref{eqn:goodwillie_T_functor} is then defined as the composite 
    \[\goodwillieTFunctor_{\sigma}\colon \udl{\func}(\udl{\sC},\udl{\D}) \xlongrightarrow{(-)^{\udl{\posetExample}}}  \udl{\func}(\udl{\sC}^{\udl{\posetExample}},\udl{\D}^{\udl{\posetExample}})\xlongrightarrow{\padding_{\sigma}^*}\udl{\func}(\udl{\sC},\udl{\D}^{\udl{\posetExample}}) \xlongrightarrow{\initialObject{\sigma}^*\puncture{\sigma}_*\puncture{\sigma}^*}\udl{\func}(\udl{\sC},\udl{\D})\]
    Observe that this admits a natural transformation \begin{equation}\label{eqn:theta_transformation}
        \theta\colon \id\rightarrow \goodwillieTFunctor_{\sigma}
    \end{equation} from the identity functor induced by the adjunction unit $\eta\colon \id\rightarrow \puncture{\sigma}_*\puncture{\sigma}^*$. In more detail, for $F\in\udl{\func}(\udl{\sC},\udl{\D})$, by virtue of the diagram
    \begin{center}
        \begin{tikzcd}
            \udl{\sC} \rar["\padding_{\sigma}"] \ar[dr, equal]& \udl{\sC}^{\udl{\posetExample}} \dar["\initialObject{\sigma}^*"]\rar["F^{\udl{\posetExample}}"] & \udl{\D}^{\udl{\posetExample}} \dar["\initialObject{\sigma}^*"]\\
            & \udl{\sC} \rar["F"] & \udl{\D},
        \end{tikzcd}
    \end{center}
    the map $\theta_F\colon F\rightarrow\goodwillieTFunctor_{\sigma}F$ is given by 
    \[\theta_F\colon F \simeq \initialObject{\sigma}^*\padding_{\sigma}^*F^{\udl{\posetExample}} \xlongrightarrow{\quad\initialObject{\sigma}^*\eta_{\padding_{\sigma}^*}\quad}\initialObject{\sigma}^*\puncture{\sigma}_*\puncture{\sigma}^*\padding_{\sigma}^*F^{\udl{\posetExample}} = \goodwillieTFunctor_{\sigma}F. \]
\end{cons}

\begin{lem}\label{lem:padding_preserves_kan_extensions_and_final_objects}
    Suppose $\udl{\sC}$ is $\sigma$--extensible and has a final object. Then the functor $\padding_{\sigma}\colon \udl{\sC}\rightarrow\udl{\func}(\udl{\posetExample},\udl{\sC})$ preserves $\sigma$--left Kan extensions and final objects.
\end{lem}
\begin{proof}
    We first show that it preserves final objects. By construction, we know that $\padding_{\sigma}\simeq \sigma_!\circ \initialObject{\sigma}_*$, and so since $\initialObject{\sigma}_*$ preserves limits, it suffices to argue that $\sigma_!\colon \udl{\func}(\subPoset{\udl{\posetExample}},\udl{\sC})\rightarrow \udl{\func}(\udl{\posetExample},\udl{\sC})$ preserves final objects. Now, the final object in $\udl{\func}(\subPoset{\udl{\posetExample}},\udl{\sC})$ is given by $\constant_{\ast_{\sC}}$. For every $\ell\in\udl{\posetExample}$, we know that 
    \[(\sigma_!\constant_{\ast_{\sC}})(\ell) \simeq \colim\big((\sigma\downarrow \ell) \longrightarrow \subPoset{\udl{\posetExample}}\rightarrow \terminalTCat \xlongrightarrow{{\ast_{\sC}}} \udl{\sC}\big)\]
    Note that $\terminalTCat$ is a groupoid, and so the map $(\sigma\downarrow\ell)\rightarrow \terminalTCat$ factors through the localisation $|\sigma\downarrow\ell|$. Because localisations are colimit cofinal we may compute and $(\sigma_!\constant_{\ast_{\sC}})(\ell)$ as a colimit over $|\sigma\downarrow \ell|$. However $|\sigma\downarrow \ell| \simeq \terminalTCat$ since $(\sigma\downarrow \ell)$ has an initial object $\distributiveInitialObject \rightarrow \ell$, and so we see that $(\sigma_!\constant_{\ast_{\sC}})(\ell) \simeq \ast_{\sC}$, as was to be shown. 
    
    Finally, to see that $\padding_{\sigma}$ preserves $\sigma$--left Kan extensions, we just have to show that the composite square in the diagram
    \begin{center}
        \begin{tikzcd}
            \udl{\func}(\udl{\posetExample},\udl{\sC}) \rar["(\id\times\initialObject{\sigma})_*"]& \udl{\func}(\udl{\posetExample}\times \subPoset{\udl{\posetExample}},\udl{\sC}) \rar["(\id\times{\sigma})_!"]& \udl{\func}(\udl{\posetExample}\times \udl{\posetExample},\udl{\sC})\\
            \udl{\func}(\subPoset{\udl{\posetExample}},\udl{\sC}) \rar["(\id\times\initialObject{\sigma})_*"]\uar["\sigma_!", hook]& \udl{\func}(\subPoset{\udl{\posetExample}}\times \subPoset{\udl{\posetExample}},\udl{\sC}) \rar["(\id\times{\sigma})_!"]\uar["(\sigma\times\id)_!", hook]& \udl{\func}(\subPoset{\udl{\posetExample}}\times\udl{\posetExample},\udl{\sC})\uar["(\sigma\times\id)_!", hook]
        \end{tikzcd}
    \end{center}
    commutes. However the right square clearly commutes, and so it suffices to show that Beck--Chevalley left square in the diagram commutes. But since $\initialObject{\sigma}_*$ is given by the final object at every point of $\subPoset{\udl{\posetExample}}$ away from $\distributiveInitialObject$, the required commutation statement is immediate from the fact that $\sigma_!$ preserves final objects, shown above.
\end{proof}

Using the constructions above, we may now provide the putative excisive approximation functor.

\begin{cons}[Excisive approximation functors]\label{cons:excisive_approximation_functors}
    Suppose $\udl{\sC}$ has final objects and $\udl{\D}$ has sequential colimits. We define the \textit{$\sigma$--excisive approximation functor} $\goodwillieApprox_{\sigma}\colon \udl{\func}(\udl{\sC},\udl{\D})\rightarrow \udl{\func}(\udl{\sC},\udl{\D})$ as 
    \[\goodwillieApprox_{\sigma}\coloneqq \colim\big(\id\xlongrightarrow{\theta} \goodwillieTFunctor_{\sigma}\xlongrightarrow{\theta\circ\goodwillieTFunctor_{\sigma}}\goodwillieTFunctor_{\sigma}\circ\goodwillieTFunctor_{\sigma}\xlongrightarrow{\theta\circ\goodwillieTFunctor_{\sigma}\circ\goodwillieTFunctor_{\sigma}}\cdots\big).\]
\end{cons}

\begin{obs}\label{obs:commuting_approximations_with_exact_functors}
    Let $F\colon \udl{\D}\rightarrow \udl{\E}$ and  $G\colon \udl{\sC}\rightarrow \udl{\D}$ be functors such that $G$ preserves $\sigma$--left Kan extensions and final objects. Then the canonical maps
    \[\goodwillieTFunctor_{\sigma}(F\circ G) \rightarrow (\goodwillieTFunctor_{\sigma}F)\circ G\quad\quad\quad\goodwillieApprox_{\sigma}(F\circ G) \rightarrow (\goodwillieApprox_{\sigma}F)\circ G\]
    are equivalences.   The second equivalence follows immediately from the first equivalence. To see that $\goodwillieTFunctor_{\sigma}(F\circ G) \rightarrow (\goodwillieTFunctor_{\sigma}F)\circ G$ is an equivalence, note that $\goodwillieTFunctor_{\sigma}(F\circ G)(-) \simeq \initialObject{\sigma}^*\sigma^*\puncture{\sigma}_*\puncture{\sigma}^*(F\circ G)(\padding_{\sigma}(-)) \simeq \initialObject{\sigma}^*\sigma^*\puncture{\sigma}_*\puncture{\sigma}^*(F\padding_{\sigma}G(-))\simeq (\goodwillieTFunctor_{\sigma}F)\circ G$ where the second equivalence comes from the hypothesis on $G$.
\end{obs}

As in ordinary functor calculus, the key nontrivial step in proving Goodwillie's formula for excisive approximations is a technical lemma of Goodwillie's, which was later streamlined by Rezk. The goal is to say that after passing to the sequential colimit as in \cref{cons:excisive_approximation_functors}, certain poset--indexed diagrams get sent to cartesian diagrams. The idea is now to use \cref{prop:squashing_preserves_cocartesianness,prop:squashing_induces_cartesianness} to sandwich in between two successive terms in the sequential colimit with cartesian diagrams. For more intuition, the reader should compare the following lemma with {\cite[Lem. 6.1.1.26]{lurieHA}}, of which our proof is a direct generalisation.

\begin{lem}[Goodwillie, Rezk]\label{lem:rezk_goodwillie_lemma}
Let $\sigma$ be an excisable poset. Let $\underline{\sC}, \udl{\D}\in\cat_{\baseCat}$ have final objects such that $\udl{\sC}$ is $\sigma$--left extensible and and $\underline{\D}$ is $\puncture{\sigma}$--right extensible. Let $F \colon \underline{\sC} \rightarrow \underline{\D}$ be a functor. Suppose $D \colon \udl{\posetExample}\rightarrow \underline{\sC}$ is $\sigma$--cocartesian. Then the canonical map $\theta_F : FD \rightarrow (\goodwillieTFunctor_{\sigma}F)(D)$ constructed above factors as  $\theta_F\colon FD\rightarrow E\rightarrow (\goodwillieTFunctor_{\sigma}F)(D)$ where  $E\in \cartesianFunc(\udl{\posetExample},\underline{\D})$.
\end{lem}
\begin{proof}
We first make three observations:
\begin{enumerate}[label=(\alph*)]
    \item By construction, the functor $\padding_{\sigma}\colon\udl{\sC}\rightarrow\udl{\sC}^{\udl{\posetExample}}$ satisfies that the adjunction unit  $\sigma^*\padding_{\sigma}\rightarrow \initialObject{\sigma}_*\initialObject{\sigma}^*\sigma^*\padding_{\sigma}$ of functors $\udl{\sC}\rightarrow \udl{\sC}^{\subPoset{\udl{\posetExample}}}$ is an equivalence.
    
    \item Let $x\in\udl{\posetExample}$. By \cref{prop:squashing_preserves_cocartesianness}, the diagram $(x\vee-)^*D\colon \udl{\posetExample}\rightarrow \udl{\sC}$ is $\sigma$--cocartesian.

    \item Consider the composition
     \begin{equation}\label{eqn:rezk_lemma_composition}
         \udl{\func}(\udl{\posetExample}, \underline{\sC}) \xlongrightarrow{F_*} \udl{\func}(\udl{\posetExample}, \underline{\D}) \xlongrightarrow{\vee^*} \udl{\func}(\udl{\posetExample}\times \udl{\posetExample}, \underline{\D})\xlongrightarrow{(\id\times\puncture{\sigma})^*}\udl{\func}(\udl{\posetExample}\times\puncture{\udl{\posetExample}}, \underline{\D}) 
     \end{equation}
     We claim that the composition \cref{eqn:rezk_lemma_composition} factors through $\cartesianFunc(\udl{\posetExample},\udl{\func}(\puncture{\udl{\posetExample}},\udl{\D}))$. This may be checked pointwise by fixing an arbitrary $x\in \puncture{\udl{\posetExample}}$. So let $Y\in\udl{\func}(\udl{\posetExample},\udl{\sC})$. In this case, we get that $((\id\times\puncture{\sigma})^*F\vee^*Y)(x)\colon \udl{\posetExample}\rightarrow\udl{\D}$ is given by the functor $FY\circ (-\vee x)$, which is indeed cartesian by \cref{prop:squashing_induces_cartesianness} as claimed. Observe that, in particular,  $(\id\times\puncture{\sigma})_*(\id\times\puncture{\sigma})^*F\vee^*Y\in\udl{\func}(\udl{\posetExample},\udl{\func}(\udl{\posetExample},\udl{\D}))$ is in fact an object in $\cartesianFunc(\udl{\posetExample},\udl{\func}(\udl{\posetExample},\udl{\D}))$, and so evaluating at $\distributiveInitialObject$ in the second $\udl{\posetExample}$ variable gives that $(\id\times\sigma\initialObject{\sigma})^*(\id\times\puncture{\sigma})_*(\id\times\puncture{\sigma})^*F\vee^*Y\in\cartesianFunc(\udl{\posetExample},\udl{\D})$.
\end{enumerate}

Now, note that for each $x\in{\udl{\posetExample}}$, we have an identity map $\initialObject{\sigma}^*\sigma^*(x\vee-)^*D\longrightarrow \initialObject{\sigma}^*\sigma^*\padding_{\sigma}(D(x))$ on the object $D(x)\in\udl{\sC}$. Thus, adjointing $\initialObject{\sigma}^*$ over, we obtain a map $\sigma^*(x\vee-)^*D\longrightarrow \initialObject{\sigma}_*\initialObject{\sigma}^*\sigma^*\padding_{\sigma}(D(x))\simeq \sigma^*\padding_{\sigma}(D(x))$, where the equivalence is by point (a). By adjointing now $\sigma^*$ over, we obtain a map $(x\vee-)^*D \simeq \sigma_!\sigma^*(x\vee-)^*D\longrightarrow \padding_{\sigma}(D(x))$ in $\udl{\sC}^{\udl{\posetExample}}$, where the equivalence is by point (b). Letting $x\in{\udl{\posetExample}}$ vary, we obtain a morphism in $\udl{\func}({\udl{\posetExample}},\udl{\sC}^{\udl{\posetExample}})$ \[(\bullet\vee-)^*D \longrightarrow \padding_{\sigma}(D(\bullet))\]  where the $\bullet$ variable corresponds to the first ${\udl{\posetExample}}$ variable. Applying $(\id\times\puncture{\sigma})_*(\id\times\puncture{\sigma})^*F$ to this morphism and precomposing with the adjunction unit gives 
\begin{equation*}
    F(\bullet\vee-)^*D\longrightarrow(\id\times\puncture{\sigma})_*(\id\times\puncture{\sigma})^*F\big((\bullet\vee-)^*D\big)\longrightarrow (\id\times\puncture{\sigma})_*(\id\times\puncture{\sigma})^*F\padding_{\sigma}(D(\bullet))
\end{equation*}
By evaluating at $\distributiveInitialObject$ in the second $\udl{\posetExample}$ variable (i.e. by applying $(\id\times \sigma\initialObject{\sigma})^*$ to this map) and defining $E(\bullet)\coloneqq (\id\times \sigma\initialObject{\sigma})^*(\id\times\puncture{\sigma})_*(\id\times\puncture{\sigma})^*F\big((\bullet\vee-)^*D\big)$, we obtain the morphism
\[FD(\bullet) \longrightarrow E(\bullet) \longrightarrow (\goodwillieTFunctor_{\sigma}F)(D)(\bullet)\]
in $\udl{\func}(\udl{\posetExample},\udl{\D})$, which is equivalent to $\theta_F$ from \cref{eqn:theta_transformation} by \cref{lem:padding_preserves_kan_extensions_and_final_objects} since $\theta_F$ was constructed using the $\puncture{\sigma}^*\dashv\puncture{\sigma}_*$ adjunction unit, which in our present two variable setting, corresponds to the adjunction unit $(\id\times\puncture{\sigma})^*\dashv(\id\times\puncture{\sigma})_*$.  Furthermore, by point (c) above, we know that $E\in\cartesianFunc(\udl{\posetExample},\udl{\D})$. This completes the proof.
\end{proof}

By passing to a sequential colimit of cartesian diagrams, we thus see that applying $P_{\sigma}$ yields a $\sigma$--excisive functor.

\begin{lem}[Goodwillie approximations are excisive, ``{\cite[Lem. 6.1.1.33]{lurieHA}}'']\label{lem:approximations_are_excisive}
Let $\sigma$ be an excisable poset. Let $\underline{\sC}, \udl{\D}\in\cat_{\baseCat}$ have final objects such that $\udl{\sC}$ is $\sigma$--left extensible and and $\underline{\D}$ is $\puncture{\sigma}$--right extensible. Let $F \colon \underline{\sC} \rightarrow\underline{\D}$ be a functor. Then $\goodwillieApprox_{\sigma}F \colon \underline{\sC} \rightarrow \underline{\D}$ is $\sigma$--excisive.
\end{lem}
\begin{proof}
Let $D \colon \udl{\posetExample} \rightarrow \underline{\sC}$ be $\sigma$--cocartesian. We want to show that $(\goodwillieApprox_{\sigma}F)(D)$ is cartesian. By definition, we have $(\goodwillieApprox_{\sigma}F)(D)\coloneqq \colim\big(FD\rightarrow (\goodwillieTFunctor_{\sigma}F)(D)\rightarrow(\goodwillieTFunctor_{\sigma}^2F)(D)\rightarrow\cdots\big)$.
By the Rezk \cref{lem:rezk_goodwillie_lemma}, we have a cartesian factorisation \[(\goodwillieTFunctor_{\sigma}^kF)(D) \rightarrow E_k \rightarrow (\goodwillieTFunctor_{\sigma}^{k+1}F)(D)\] so that we could alternatively have computed $(\goodwillieApprox_{\sigma}F)(D)$ as a sequential colimit of $E_0\rightarrow E_1\rightarrow E_2\rightarrow\cdots$.
 Since each $E_i$ is cartesian in $\underline{\D}$, we thus obtain by $\sigma$--differentiability of $\udl{\D}$ that the sequential colimit $(\goodwillieApprox_{\sigma}F)(D)$ is cartesian, as required.
\end{proof}

\begin{lem}[Idempotence of Goodwillie approximations, ``{\cite[Lem. 6.1.1.35]{lurieHA}}'']\label{lem:approximations_are_idempotent}
Let $\sigma$ be an excisable poset. Let $\underline{\sC}, \udl{\D}\in\cat_{\baseCat}$ have final objects such that $\udl{\sC}$ is $\sigma$--left extensible and  $\underline{\D}$ is $\sigma$--differentiable. Let $F \colon \underline{\sC} \rightarrow\underline{\D}$ be a functor. Then $\goodwillieApprox_{\sigma}\theta \colon \goodwillieApprox_{\sigma}F \rightarrow \goodwillieApprox_{\sigma}\goodwillieTFunctor_{\sigma}F$ is an equivalence. Consequently, the canonical transformation $\goodwillieApprox_{\sigma}F \rightarrow\goodwillieApprox_{\sigma}\goodwillieApprox_{\sigma}F$ is an equivalence.
\end{lem}
\begin{proof}
By $\sigma$--differentiability of $\udl{\D}$, the functor $\goodwillieApprox_{\sigma}\colon \udl{\func}(\udl{\sC},\udl{\D})\rightarrow \udl{\func}(\udl{\sC},\udl{\D})$ commutes with $\sigma$--right Kan extensions, and so  the canonical map \[\goodwillieApprox_{\sigma}\goodwillieTFunctor_{\sigma}F \longrightarrow \initialObject{\sigma}^*\sigma^*\puncture{\sigma}_*\puncture{\sigma}^*\big(\goodwillieApprox_{\sigma}(F\circ \padding_{\sigma})\big)\xrightarrow{\simeq} \initialObject{\sigma}^*\sigma^*\puncture{\sigma}_*\puncture{\sigma}^*\big((\goodwillieApprox_{\sigma}F)\circ \padding_{\sigma}\big) = \goodwillieTFunctor_{\sigma}\goodwillieApprox_{\sigma}F\] is an equivalence, where the second equivalence is by \cref{lem:padding_preserves_kan_extensions_and_final_objects} and \cref{obs:commuting_approximations_with_exact_functors}. But then $\goodwillieApprox_{\sigma}F$ was $\sigma$--excisive by \cref{lem:approximations_are_excisive}, and so since $\padding_{\sigma}\in\udl{\func}(\udl{\sC},\cocartesianFunc{\sigma}(\udl{\posetExample},\udl{\sC}))$, we obtain that the canonical map 
 $\goodwillieApprox_{\sigma}F\rightarrow\initialObject{\sigma}^*\sigma^*\puncture{\sigma}_*\puncture{\sigma}^*\big((\goodwillieApprox_{\sigma}F)\circ \padding_{\sigma}\big)$ is an equivalence. Therefore, by the commuting triangle
 \begin{center}
     \begin{tikzcd}
         \goodwillieApprox_{\sigma}F \rar["\goodwillieApprox_{\sigma}\theta_F"]\ar[dr,"\canonical"', "\simeq"]& \goodwillieApprox_{\sigma}\goodwillieTFunctor_{\sigma}F \dar["\canonical","\simeq"']\\
         & \initialObject{\sigma}^*\sigma^*\puncture{\sigma}_*\puncture{\sigma}^*\big((\goodwillieApprox_{\sigma}F)\circ \padding_{\sigma}\big)
     \end{tikzcd}
 \end{center}
 we get that $\goodwillieApprox_{\sigma}\theta_F$ is an equivalence, as required.
\end{proof}

We are at last ready to state and prove \cref{alphthm:excisive_approximation_left_adjoint}.

\begin{thm}[Excisive Bousfield localisation]\label{mainthm:excisive_approximation_left_adjoint}
    Let $\sigma\colon \subPoset{\udl{\posetExample}}\subseteq \udl{\posetExample}$ be an excisable poset. Let $\udl{\sC}$ have a final object and $\udl{\D}$ is $\sigma$--differentiable. Then the fully faithful inclusion $\udl{\excisiveCat}^{\sigma}(\udl{\sC},\udl{\D})\subseteq \udl{\func}(\udl{\sC},\udl{\D})$ admits a left adjoint given by $\goodwillieApprox_{\sigma}$. Furthermore, if $\udl{\D}$ has finite parametrised limits and they commute with sequential colimits, then $\goodwillieApprox_{\sigma}$ is even parametrised left exact.
\end{thm}
\begin{proof}
By \cref{lem:approximations_are_excisive}, the endofunctor $\goodwillieApprox_{\sigma}$ on $\udl{\func}(\udl{\sC},\udl{\D})$ factors through $\udl{\excisiveCat}^{\sigma}(\udl{\sC},\udl{\D})$. On the other hand, if $F\in\udl{\func}(\udl{\sC},\udl{\D})$ were $\sigma$--excisive, then since $\padding_{\sigma}\in\udl{\func}(\udl{\sC},\cocartesianFunc{\sigma}(\udl{\posetExample},\udl{\sC}))$,  the map $F \rightarrow \goodwillieTFunctor_{\sigma}F$ is an equivalence (which in particular implies that $\goodwillieApprox_{\sigma}F$ is also $\sigma$--excisive). Iterating this observation, we get that $F \rightarrow \goodwillieApprox_{\sigma}F$ is an equivalence.

To see that it is a Bousfield localisation we just need to show that  the two maps
\[\goodwillieApprox_{\sigma}\theta_F,\: \theta_{\goodwillieApprox_{\sigma}F} \colon \goodwillieApprox_{\sigma}F \longrightarrow \goodwillieApprox_{\sigma}\goodwillieApprox_{\sigma}F\] are equivalences. The case of $\goodwillieApprox_{\sigma}\theta_{F}$ is covered already by \cref{lem:approximations_are_idempotent}, whereas that of $\theta_{\goodwillieApprox_{\sigma}F}$ is also done since $\goodwillieApprox_{\sigma}F$ was $\sigma$--excisive. This completes the proof of the first part. The second statement about parametrised left exactness of $\goodwillieApprox_{\sigma}$ under the extra hypothesis on $\udl{\D}$ is immediate from the explicit construction of $\goodwillieApprox_{\sigma}$ given in \cref{cons:excisive_approximation_functors}.
\end{proof}

The next result is a generalisation of the Goodwillie calculus fact that $(n-1)$--excisive functors are automatically also $n$--excisive, and for it, recall the induced excisable poset structure $\sigma_a$ associated to an object $a\in\udl{\posetExample}$ in an excisable poset $\sigma\colon \subPoset{\udl{\posetExample}}\subseteq \udl{\posetExample}$.

\begin{prop}\label{prop:excision_implications_along_inclusions_of_orbits}
    Let $\sigma$ be an excisable poset, $a\in\udl{\posetExample}$ an object, and $F\colon \udl{\sC}\rightarrow\udl{\D}$ a $\sigma_a$--excisive functor. Then $F$ is also $\sigma$--excisive.
\end{prop}
\begin{proof}
    Let $X\colon \udl{\posetExample}\rightarrow\udl{\sC}$ be a $\sigma$--cocartesian diagram. By \cref{prop:faces_of_strong_cocartesian}, we know that all $a$--faces of $X$ are $\sigma_a$--cocartesian. Hence, by hypothesis, every $a$--face of $FX\colon \udl{\posetExample}\rightarrow\udl{\D}$ is $\sigma_a$--cartesian. By \cref{prop:face_of_cartesian}, this implies that $FX$ is cartesian, as required.
\end{proof}

\begin{rmk}
    As a consequence of \cref{mainthm:excisive_approximation_left_adjoint} and \cref{prop:excision_implications_along_inclusions_of_orbits}, we obtain a natural transformation $\goodwillieApprox_{\sigma}\rightarrow \goodwillieApprox_{\sigma_a}$ for every object $a\in\udl{\posetExample}$. In particular, if we have a countable number of excisable posets equipped with an object $(\sigma_i\colon \subPoset{\udl{\posetExample}}_i\subseteq \udl{\posetExample}_i, a_i\in\udl{\posetExample}_i)$ such that $\sigma_i\cong \sigma_{i+1,a_{i+1}}$, then will obtain a Goodwillie tower 
    \[\id\longrightarrow \cdots \rightarrow \goodwillieApprox_{\sigma_2} \rightarrow \goodwillieApprox_{\sigma_1}\rightarrow \goodwillieApprox_{\sigma_0}\]
    associated to this sequence of excisable posets.
\end{rmk}

\section{Parametrised cubical excision}\label{section:parametrised_cubical_excision}
Our goal is to employ the abstract theory of \cref{section:excision_for_posets} to prove \cref{alphthm:spherically_faithful_implies_semiadditivity} and \cref{alphThm:sphere_invertible_equivalent_to_parametrised_stable} by specialising to the case of atomic orbital base categories. We first recall the cubical constructions from \cite[$\S3$]{kaifNoncommMotives} in \cref{subsection:recollections_parametrised_cubes} and further develop some basic properties, a crucial one being \cref{prop:productsOfStronglyCocartesianCubes} that external indexed products of singleton cocartesian cubes are again singleton cocartesian; using these cubes, we construct the parametrised spheres that will generalise the regular representation spheres from the equivariant context in \cref{subsection:parametrised_spheres}; we then justify in \cref{subsec:cubical_excision} that we may port the theory of \cref{section:excision_for_posets} to the present context of cubes  before proceeding to prove \cref{alphthm:spherically_faithful_implies_semiadditivity} and \cref{alphThm:sphere_invertible_equivalent_to_parametrised_stable} in \cref{subsec:semiadditivity_stability}. As before,  recall the notations and conventions from \cref{subsec:notations}.

\subsection{Cubes}
\label{subsection:recollections_parametrised_cubes}

\begin{defn}[{\cite[Def. 4.1]{nardinExposeIV}}]\label{DefinitionAtomicOrbital}
Let $\baseCat$ be a small category. We say that it is  \textit{atomic} if whenever we have $f : W\rightarrow V$ and $g : V \rightarrow W$ in ${\baseCat}$ such that $g\circ f$ is an equivalence, then $f$ and $g$ were already inverse equivalences; and we say that it is \textit{orbital} if the finite coproduct cocompletion $\finite_{\baseCat}$ admits finite pullbacks. 
\end{defn}

Throughout this section,  $\basecat$ will always be an atomic orbital category unless otherwise specified. For examples of such categories, we refer the reader to \cref{example:atomic_orbital}. We first repeat some cubical constructions and facts established in \cite[\textsection 3]{kaifNoncommMotives} for the reader's convenience. 

\begin{nota}
    We write $\underline{\Delta}^1\in \udl{\cat}$ for the constant parametrised category with value $\Delta^1$, and we will often be vague about the base category we are parametrising over since we will often need to change base categories when we consider $\udl{\Delta}^1$.
\end{nota}

\begin{defn}\label{defn:parametrised_cubes}
    Let $w\colon W \rightarrow V$ be a morphism in $\in{\finite}_{\baseCat}$ and let $\underline{\sC}$ be a $\baseCat$--category.  By \textit{the parametrised $w$--cube} we will mean the $\baseCat_{/V}$--category $w_*\underline{\Delta}^1$ and by  \textit{a parametrised $w$--cube in $\underline{\sC}$} we will mean a $\baseCat_{/V}$--functor $Q\colon w_*\underline{\Delta}^1\rightarrow V^*\underline{\sC}$. 
\end{defn}
%\todo[inline]{K: I've been thinking a lot about this to no avail actually. Do we know that these cubes are parametrised finite categories (ie can they be built in finitely many steps by finite indexed coproducts and pushouts)? I tried to think of strongly finiteness of parametrised posets and alternatively, tried to use that $\func(f_*\sC,\D)=f_{\otimes}\func(\sC,\D)$ for presentable $\D$ before, but  I failed}

\begin{obs}\label{obs:top_points_in_cubes}
    An important feature of the parametrised cubes $w_*\udl{\Delta}^1$ is that the only global points they have are the initial and final objects $\distributiveInitialObject$ and $\distributiveFinalObject$ since a map $\terminalTCat\rightarrow w_*\udl{\Delta}^1$ corresponds by adjunction to a map $\terminalTCat\rightarrow \udl{\Delta}^1$. The other parametrised points of $w_*\udl{\Delta}^1$ will thus come from $B$--points $b_!\terminalTCat\rightarrow w_*\udl{\Delta}^1$ where $b\colon B\rightarrow V$ is the unique map in $\basecat_{/V}$ (where $b\in\basecat_{/V}$  is not the final object).
\end{obs}

\begin{fact}[{\cite[Prop. 3.1.2]{kaifNoncommMotives}}]\label{prop:parametrisedCubesAreFibrewiseCubes}
    The parametrised cubes $w_*\underline{\Delta}^1$ are all parametrised posets, i.e. they belong to the full subcategory $\func(\baseCat_{/V}\op,\mathrm{Poset})\subseteq \func(\baseCat_{/V}\op,\cat)$. In fact, these are fibrewise given by cubes of various dimensions.    
\end{fact}

\begin{cons}[Points in cubes are subsets]\label{cons:pointsInCubesAreSubsets}
    Suppose $\basecat$ has a final object $V$ and $w\colon W \rightarrow V$ is the unique map to the final object. Let $B\in\basecat$ and $b\colon B \rightarrow V$ the unique map. We would like to describe a $B$--point $y\colon b_!\terminalTCat\rightarrow w_*\underline{\Delta}^1$. First of all, write $b^*W = U_1\amalg \cdots \amalg U_n$ for the decomposition of $b^*W$ into orbits over $B$ and $u_i\colon U_i\rightarrow B$ for the structure maps. Adjoining over $y$ we get a map
    \[\terminalTCat \longrightarrow b^*w_*w^*\underline{\Delta}^1 \simeq u_{1*}u_1^*\underline{\Delta}^1\times\cdots\times u_{k*}u_k^*\underline{\Delta}^1\]
    Adjoining over each of the component yields maps
    \[u_1^*\terminalTCat\longrightarrow u_1^*\underline{\Delta}^1\] which is either the constant functor at 0 or 1. Hence the $B$--point $y$ can be viewed as a choice of orbits in the decomposition  $b^*W = U_1\amalg \cdots \amalg U_n$. In general, given any map $w\colon W\to V$ in $\basecat$ where $V$ is not necessarily final, we may apply the construction above in $\basecat_{/V}$.
\end{cons}

The following is the key construction in the theory of parametrised cubes.

\begin{cons}[Singleton inclusion, {\cite[Cons. 3.1.5]{kaifNoncommMotives}}]\label{cons:singletonInclusion}
    Let $w\colon W\rightarrow V$ be a map in $\underline{\finite}_{\baseCat}$. We would like to construct a map
    \begin{equation}\label{eqn:fan_inclusion}
        \varphi_w \colon (w_!\terminalTCat)\tcone\longrightarrow w_*\underline{\Delta}^1
    \end{equation}
    which generalises the inclusion of the subsets of size at most 1 in Goodwillie's definition of strong cocartesianness.  For this, we first construct a map
    \[\psi_w\colon w_!\terminalTCat \longrightarrow w_*\underline{\Delta}^1.\]
    First note by atomic orbitality that we have the pullback in $\finite_{\baseCat}$
    \begin{equation}\label{eqn:singletonInclusionPullback}
        \begin{tikzcd}
            W\coprod C \ar[dr, phantom , very near start , "\lrcorner"]\rar["\id\coprod c"]\dar["\id\coprod \overline{c}"'] & W\dar["w"]\\
            W\rar["w"'] & V
        \end{tikzcd}
    \end{equation}
    where $C$ is some object in $\finite_\baseCat$. In particular, we have the decomposition
    \[w^*w_*\underline{\Delta}^1\simeq \id_*\id^*\underline{\Delta}^1\times \overline{c}_*c^*\underline{\Delta}^1\simeq \underline{\Delta}^1\times \overline{c}_*c^*\underline{\Delta}^1 \]
    Now by adjunction, to construct $\psi_w$ it suffices to construct its adjoint $\overline{\psi}_w\colon \terminalTCat \rightarrow w^*w_*\underline{\Delta}^1$. By the decomposition above, this is equivalent to constructing maps 
    \[\terminalTCat \rightarrow \underline{\Delta}^1\quad\text{and}\quad \big(\terminalTCat\rightarrow \overline{{c}}_*c^*\underline{\Delta}^1\big) \Leftrightarrow\big(\overline{c}^*\terminalTCat\simeq c^*\terminalTCat \rightarrow c^*\underline{\Delta}^1\big)\] To this end, we declare the first map to be the inclusion of the target $\{1\} \rightarrow \Delta^1$ and the second map to be $c^*$ applied to the inclusion of the source $\{0\} \rightarrow \Delta^1$. This yields the map $\psi_w$, which one should think of the inclusion of the singletons in a cube. By \cite[Cons. 3.1.3]{kaifNoncommMotives}, the map $\psi_w$ constructed above induces the desired map \cref{eqn:fan_inclusion}.
  
\end{cons}

\begin{rmk}\label{rmk:classical_cubes}
The maps $\varphi_w$ generalise the inclusion of those subsets of size at most 1 into a cube. Namely if $V$ is a final object of $\basecat$ and $W$ is equivalent to a finite coproduct $\coprod_n V$ living over $V$ via the fold map, then $w_*\udl{\Delta}^1$ is the constant $\basecat$-category on the category $\prod_n \Delta^1$ and the map $\varphi_w$ is precisely given levelwise by the standard inclusion of subsets of size at most 1. See \cite[Ex. 3.1.6]{kaifNoncommMotives} for more details.
\end{rmk}

\begin{fact}\label{fact:basics_of_singleton_inclusion}
    Here are some basic properties of the construction above.
    \begin{enumerate}[label=(\arabic*)]
        \item The singleton inclusion map is stable under basechange \cite[Rmk. 3.1.7]{kaifNoncommMotives} in that for $b\colon B \rightarrow V$ a map in $\basecat$, writing $\overline{w}\colon b^*W\rightarrow B$ for the basechange of $w\colon W \rightarrow V$ along $b$, we get  $b^*\varphi_w \simeq \varphi_{\overline{w}}$.

        \item The map $\varphi_w\colon (w_!\underline{\ast})\tcone \rightarrow w_*\underline{\Delta}^1$ is fully faithful by \cite[Cor. 3.1.8]{kaifNoncommMotives}.
    \end{enumerate}
\end{fact}

\begin{terminology}
    It will be convenient to use also the phrase \textit{singleton cocartesian} to denote any parametrised cube which is left Kan extended from the singletons, i.e. those which are $\varphi_w$--cocartesian in the sense of \cref{defn:sigma_cocartesian}. This allows us to be vague about which $w$ we are using. Analogously, we will use the term \textit{singleton cartesian} for the obvious dual notion. More precisely, a cubical diagram is singleton cartesian if it is right Kan extended along \[(w_!\ast)\tcocone \simeq ((w_!\ast)\tcone)\op\xhookrightarrow{\varphi_w\op} (w_*\udl{\Delta}^1)\op \simeq w_*\udl{\Delta}^1.\]
\end{terminology}

Having set up the basic language, our next task is to further develop the theory by proving that the external indexed product of any singleton cocartesian cube is again singleton cocartesian in \cref{prop:productsOfStronglyCocartesianCubes}. This seems to us an under--appreciated fact in ordinary Goodwillie calculus since the inductive proof is relatively painless there. We have no  recourse to such inductive proofs in this setting, and so the proof we provide will be necessarily categorical and ``coordinate--free''. 

To this end, we first formalise the process of taking the ``external indexed product'' of diagrams and record a basic property of this construction.
\begin{cons}[External indexed products of diagrams]\label{cons:indexedProductComparisonInternalHom}
    Let $a \colon V \rightarrow A$ be a map in $\basecat$. Let $\underline{I}, \underline{\sC} \in \cat_{\basecat_{/V}}$. We now construct a natural map
    \[\varepsilon^*\colon a_*\underline{\func}(\underline{I},\underline{\sC}) \longrightarrow \underline{\func}(a_*\underline{I},a_*\underline{\sC})\] which, informally, takes a diagram $X\colon \udl{I}\rightarrow\udl{\sC}$ and produces the new diagram $a_*X=\varepsilon^*X \colon a_*\udl{I}\rightarrow a_*\udl{\sC}$. This is achieved as follows: let $\varepsilon\colon a^*a_*\underline{I}\rightarrow \underline{I}$ be the adjunction counit. We then obtain a map
    \[\underline{\func}(\underline{I},\underline{\sC}) \xlongrightarrow{\varepsilon^*} \underline{\func}(a^*a_*\underline{I},\underline{\sC})\]
    Applying $a_*$ to this map and noting that $a_*\underline{\func}(a^*a_*\underline{I},\underline{\sC})\simeq \underline{\func}(a_*\underline{I},a_*\underline{\sC})$, we obtain the desired map. 

    Now suppose that $\underline{I}$ has an initial object. Observe that by atomicity of $\basecat$, the adjunction counit $\varepsilon\colon a^*a_*\underline{I}\rightarrow\underline{I}$ is given by projecting onto the $\underline{I}$ component in the decomposition $a^*a_*\underline{I} \simeq \underline{I}\times c_*c^*\underline{I}$ associated to the decomposition $V\times_AV \simeq V\amalg C$. Since $c_*c^*\underline{I}$ has an initial object, we have an adjunction $\distributiveInitialObject \colon \terminalTCat \rightleftharpoons c_*c^*\underline{I}\cocolon \pi$ where $\pi$ is the unique functor. Taking the product with $\underline{I}$ then yields the adjunction
    \[\distributiveInitialObject\colon \underline{I}\rightleftharpoons a^*a_*\underline{I} \cocolon \varepsilon\]
\end{cons}

\begin{prop}\label{prop:leftKanExtensionIndexedProductInterchange}
    Let $j\colon \underline{I}\rightarrow\underline{J}$ be a functor which preserves the initial object. Then the square
    \begin{center}
        \begin{tikzcd}
            a_*\underline{\func}(\underline{I},\underline{\sC})\dar["j_!"']\rar["\varepsilon^*"] & \underline{\func}(a_*\underline{I},a_*\underline{\sC})\dar["(a_*j)_!"']\\
            a_*\underline{\func}(\underline{J},\underline{\sC}) \rar["\varepsilon^*"] & \underline{\func}(a_*\underline{J},a_*\underline{\sC})
        \end{tikzcd}
    \end{center}
    commutes.
\end{prop}
\begin{proof}
    From the adjunction $\distributiveInitialObject\dashv \varepsilon$ of \cref{cons:indexedProductComparisonInternalHom} together with the fact that indexed products preserve adjunctions \cite[Lem. 4.3.2]{kaifPresentable}, we obtain an adjunction $\varepsilon^*\colon a_*\underline{\func}(\underline{I},\underline{\sC}) \rightleftharpoons\underline{\func}(a_*\underline{I},a_*\underline{\sC}) \cocolon \distributiveInitialObject^*$. Now passing everything in sight to their right adjoints in the square of the proposition yields the square
    \begin{center}
        \begin{tikzcd}
            a_*\underline{\func}(\underline{I},\underline{\sC}) & \underline{\func}(a_*\underline{I},a_*\underline{\sC})\lar["\distributiveInitialObject^*"']\\
            a_*\underline{\func}(\underline{J},\underline{\sC}) \uar["j^*"]& \underline{\func}(a_*\underline{J},a_*\underline{\sC})\uar["(a_*j)^*"]\lar["\distributiveInitialObject^*"']
        \end{tikzcd}
    \end{center}
    which clearly commutes since 
    \begin{center}
        \begin{tikzcd}
            \underline{I} \rar["\distributiveInitialObject"]\dar["j"'] & a^*a_*\underline{I} \dar["a^*a_*j"] \\
            \underline{J} \rar["\distributiveInitialObject"]& a^*a_*\underline{J} 
        \end{tikzcd}
    \end{center}
    commutes by our assumption that $j$ preserves the initial object. Hence the left adjoint square commutes too, as required.
\end{proof}

\begin{cons}[Pushforward of singletons] \label{cons:pushforwardOfSingeltons}
    Let $w\colon W \rightarrow V$ be in $\finite_\basecat$ and $a\colon V \rightarrow A$ be in $\basecat$. Recall the map $\distributiveInitialObject \colon (w_!\terminalTCat)\tcone\rightarrow a^*a_*(w_!\terminalTCat)\tcone$ given in \cref{cons:indexedProductComparisonInternalHom}. We would like to construct a factorisation of it that looks like
    \begin{center}
        \begin{tikzcd}
            &  a^*(a_!w_!\terminalTCat)\tcone\dar["a^*\theta"]\\
            (w_!\terminalTCat)\tcone\rar["\distributiveInitialObject"']\ar[ur, "z"] &  a^*a_*(w_!\terminalTCat)\tcone
        \end{tikzcd}
    \end{center}
    By construction all functors in the triangle above will preserve initial objects and so it suffices to construct the factorisation 
    \begin{center}
        \begin{tikzcd}
            &  a^*a_!w_!\terminalTCat\dar["a^*\overline{\theta}"]\\
            w_!\terminalTCat\rar["\distributiveInitialObject"']\ar[ur, "\overline{z}"] &  a^*a_*(w_!\terminalTCat)\tcone, 
        \end{tikzcd}
    \end{center}
    which we may then extend uniquely to a triangle of the first kind. We simply set $\overline{z}$ to be the adjunction unit, then by adjunction there is a unique $\overline{\theta}$ making the triangle commute. We use this choice, thus completing the construction.

    Crucially, the map $\theta$ interacts well with the maps
    \[
    \varphi_w\colon (w_!\terminalTCat)\tcone \rightarrow w_*\underline{\Delta}^1
    \] of \cref{cons:singletonInclusion}, by which we mean that the composite
    \[(a_!w_!\terminalTCat)\tcone \xlongrightarrow{\theta} a_*(w_!\terminalTCat)\tcone\xlongrightarrow{a_*\varphi_w} a_*w_*\underline{\Delta}^1\] is equivalent to $\varphi_{aw}$. To prove this, we observe that since all three functors preserve the initial object, it suffices to show that the composite 
    \[a_!w_!\terminalTCat\xlongrightarrow{\overline{\theta}} a_*(w_!\terminalTCat)\tcone\xlongrightarrow{a_*\varphi_w} a_*w_*\underline{\Delta}^1\]
    is equivalent to $\psi_{aw}$, the singleton inclusion map. Using the orbital decomposition $V\times_AV\simeq V\amalg C$ and adjoining the $a_!$ over, this composite yields the composite
    \begin{center}
        \begin{tikzcd}
            w_!\terminalTCat\rar["\distributiveInitialObject"]\dar[equal] & a^*a_*(w_!\terminalTCat)\tcone\ar[rr,"a^*a_*\varphi_w"] \dar[equal]&  &a^*a_*w_*\underline{\Delta}^1\dar[equal]\\
            w_!\terminalTCat\rar["\mathrm{incl}\times \distributiveInitialObject"] & (w_!\terminalTCat)\tcone\times c_*c^*(w_!\terminalTCat)\tcone\ar[rr,"\varphi_w\times c_*c^*\varphi_w"] & & w_*\underline{\Delta}^1\times c_*c^*w_*\underline{\Delta}^1.
        \end{tikzcd}
    \end{center}
    Hence, the datum of this composite is equivalent to the data of the maps
    \[w_!\terminalTCat \xlongrightarrow{\psi_w}w_*\underline{\Delta}^1\quad\quad\quad\quad w_!\terminalTCat\xlongrightarrow{\distributiveInitialObject} c_*c^*w_*\underline{\Delta}^1\]
    This is precisely the data of maps defining the map $\psi_{aw} \colon a_!w_!\terminalTCat \rightarrow a_*w_*\underline{\Delta}^1$.
\end{cons}

We are now ready to show that indexed products preserve singleton cocartesian cubes.

\begin{prop}\label{prop:productsOfStronglyCocartesianCubes}
    Suppose $X\colon w_*\underline{\Delta}^1\rightarrow\underline{\sC}$ is a singleton cocartesian $w$--cube and let $a\colon V\rightarrow A$ be in $\basecat$. In this case, $a_*X\colon a_*w_*\underline{\Delta}^1\rightarrow a_*\underline{\sC}$ is then a singleton cocartesian $a\circ w$--cube.
\end{prop}
\begin{proof}
    By the triangle in \cref{cons:pushforwardOfSingeltons}, we have a commuting triangle
    \begin{center}
        \begin{tikzcd}
            &  \underline{\func}(a^*(a_!w_!\terminalTCat)\tcone, \underline{\sC})\ar[dl, "z^*"']\\
            \underline{\func}((w_!\terminalTCat)\tcone,\underline{\sC}) &  \underline{\func}(a^*a_*(w_!\terminalTCat)\tcone,\underline{\sC})\uar["(a^*\theta)^*"']\lar["\distributiveInitialObject^*"]
        \end{tikzcd}
    \end{center}
    Passing to left adjoints and using the $\distributiveInitialObject\dashv \varepsilon$ adjunction, we obtain the commuting triangle 
    \begin{center}
        \begin{tikzcd}
            &  \underline{\func}(a^*(a_!w_!\terminalTCat)\tcone, \underline{\sC})\dar["(a^*\theta)_!"]\\
            \underline{\func}((w_!\terminalTCat)\tcone,\underline{\sC}) \ar[ur, "z_!"]\rar["\varepsilon^*"']&  \underline{\func}(a^*a_*(w_!\terminalTCat)\tcone,\underline{\sC}).
        \end{tikzcd}
    \end{center}
    Applying $a_*$ to this triangle and combining the result with the square from \cref{prop:leftKanExtensionIndexedProductInterchange}, we obtain the following big commuting diagram
    \begin{center}
        \begin{tikzcd}
            &  \underline{\func}((a_!w_!\terminalTCat)\tcone, a_*\underline{\sC})\dar["\theta_!"]\ar[dd, bend left = 80, "(\varphi_{aw})_!"]\\
            a_*\underline{\func}((w_!\terminalTCat)\tcone,\underline{\sC})\dar["(\varphi_w)_!"']\rar["\varepsilon^*"]\ar[ur, "z_!"] & \underline{\func}(a_*(w_!\terminalTCat)\tcone,a_*\underline{\sC})\dar["(a_*\varphi_w)_!"]\\
            a_*\underline{\func}(w_*\underline{\Delta}^1,\underline{\sC}) \rar["\varepsilon^*"] & \underline{\func}(a_*w_*\underline{\Delta}^1,a_*\underline{\sC})
        \end{tikzcd}
    \end{center}
    where the equivalence of the right vertical composition with the bent arrow is by the second paragraph from \cref{cons:pushforwardOfSingeltons}.

    Now $X$ being a singleton cocartesian $w$--cube means that there is a $Y\colon (w_!\terminalTCat)\tcone\rightarrow\underline{\sC}$ such that $X\simeq (\varphi_w)_!Y$. Therefore, by the big diagram, we see that 
    \[a_*X=\varepsilon^*X \simeq (\varphi_{aw})_!z_!Y \colon a_*w_*\underline{\Delta}^1\longrightarrow a_*\underline{\sC}\]
    as desired.
\end{proof}

\subsection{Suspensions and spheres}\label{subsection:parametrised_spheres}

In this subsection, we use the parametrised cubes to construct parametrised suspension and loop functors as well as the parametrised spheres that will feature in \cref{alphthm:spherically_faithful_implies_semiadditivity} and \cref{alphThm:sphere_invertible_equivalent_to_parametrised_stable}.

% \begin{obs}\label{obs:slices_in_atomic_orbital}
%     For any $u\in \basecat$, we have that
%     \[(\basecat_{/u})_{u/}\simeq \ast\]
%     This is because an object in this category is a morphism $z\xrightarrow{r} u$ equipped with a morphism $u\xrightarrow{i} z$ such that $r\circ i \simeq \id_u$. Hence by atomic orbitality, we must have that $r$ and $i$ were equivalences, whence the result.
% \end{obs}

For the reader's convenience, we record the following result of Shah about general base categories $\basecat$.

\begin{prop}[{\cite[Prop. 5.5]{shahThesis}}]\label{prop:colimtsInCofreeCategories}
    Let $\underline{I}\in\cat_{\basecat}$, $p\colon \underline{I}\rightarrow\terminalTCat$ the unique map, and $u\in\basecat$. Let $\varphi\in\func_{\basecat}(\underline{I},\underline{\spc}_*)$. Then $(p_!\varphi)(u) \simeq \colim_{I_u}(I_u\rightarrow \spc_*)$.  This holds similarly when we replace $\udl{\spc}_*$ with $\udl{\cat}$.
\end{prop}

\begin{cons}[Parametrised suspensions, loops, and spheres]\label{cor:parametrisedSuspensions}
    Let $\udl{\sC}$ be a parametrised cocomplete category with a final object. Let $w\colon W\rightarrow V$ be a map in $\finite_{\basecat}$ and write $\initialObject{\sigma}_w\colon \terminalTCat\hookrightarrow w_*\Delta^1$ for the inclusion of the initial object. Then we define the \textit{parametrised suspension} endofunctor $\Sigma^w$ on $\udl{\sC}$ as the composite of $\basecat_{/V}$--functors
    \[\Sigma^w\colon V^*\udl{\sC} \xhookrightarrow{\initialObject{\sigma}_{w*}} \underline{\func}(u_*\underline{\Delta}^1\backslash\distributiveFinalObject, V^*\udl{\sC})\xlongrightarrow{\colim}V^*\udl{\sC}.\] If, additionally, $\udl{\sC}$ is pointed and parametrised complete, then we also define the \textit{parametrised loop} endofunctor as the composite of $\basecat_{/V}$--functors
     \[\Omega^w\colon V^*\udl{\sC} \xhookrightarrow{(\initialObject{\sigma}_{w}\op)_!} \underline{\func}(w_*\underline{\Delta}^1\backslash\distributiveInitialObject, V^*\udl{\sC})\xlongrightarrow{\lim}V^*\udl{\sC}.\]
     In the special case when $\udl{\sC}=\udl{\spc}_*$, we define the \textit{$w$--parametrised sphere} to be 
     \[S^w\coloneqq \Sigma^wS^0\in \udl{\spc}_*.\]
     This terminology is justified since $S^w\in\func((\basecat_{/V})\op,\spc_*)$ is given fibrewise by spheres of various dimensions. This is because $w_*\underline{\Delta}^1\backslash\distributiveFinalObject$ is fibrewise given by punctured cubes by \cref{prop:parametrisedCubesAreFibrewiseCubes}, and so applying \cref{prop:colimtsInCofreeCategories}, we see that $S^w$ is fibrewise given by some suspension of $S^0$. 
\end{cons}

As one may expect, the suspension and loop functors participate in an adjunction.

\begin{prop}\label{prop:suspension_loop_adjunction}
    Let $\basecat$ be atomic orbital. For every morphism $w \colon W \rightarrow V$ in $\finite_{\basecat}$, and for $\udl{\sC}$ a pointed parametrised bicomplete category, there is a suspension--loop adjunction $\Sigma^w\colon V^*\udl{\sC}\rightleftharpoons V^*\udl{\sC} \cocolon \Omega^w$.
\end{prop}
\begin{proof}
    Let $X,Y\in V^*\udl{\sC}$. Write $\sigma_X\colon w_*\Delta^1\backslash\distributiveFinalObject\rightarrow  V^*\udl{\sC}$ and $\lambda_Y\colon w_*\Delta^1\backslash0\rightarrow V^*\udl{\sC}$ for the functors defining $\Sigma^wX$ and $\Omega^wY$ respectively. Since $\myuline{\map}_{ V^*\udl{\sC}}(0,Y)\simeq 0\simeq \myuline{\map}_{ V^*\udl{\sC}}(X,0)\in\udl{\spc}_*$, the claimed adjunction is then a  consequence of the  equivalences
    \[\myuline{\map}(\colim_{u_*\Delta^1\backslash\distributiveFinalObject}\sigma_X,Y)\simeq \lim_{u_*\Delta^1\backslash\distributiveInitialObject}\myuline{\map}(\sigma_X,Y)\simeq \lim_{u_*\Delta^1\backslash\distributiveInitialObject}\myuline{\map}(X,\lambda_Y)\simeq \myuline{\map}(X,\Omega^uY)\]
    which are clearly natural in $X$ and $Y$.
\end{proof}

\begin{defn}\label{defn:spherically_faithful_category}
    Let $\udl{\sC}$ be a pointed parametrised bicomplete category and $w \colon W \rightarrow V$ a morphism in $\finite_{\basecat}$. We say that it is:
    \begin{enumerate}[label=(\arabic*)]
        \item \textit{$w$--spherically faithful} if the endofunctor $\Sigma^w\colon V^*\udl{\sC}\rightarrow V^*\udl{\sC}$ of $\basecat_{/V}$--categories is fully faithful and $\Omega^w$ commutes with sequential colimits,
        \item \textit{$w$--spherically invertible} if the endofunctor $\Sigma^w\colon V^*\udl{\sC}\rightarrow V^*\udl{\sC}$ of $\basecat_{/V}$--categories is  an equivalence.
    \end{enumerate}
     We say that $\udl{\sC}$ is \textit{spherically faithful} (resp. \textit{spherically invertible}) if $\udl{\sC}$ is $w$--spherically faithful (resp. $w$--spherically invertible) for all morphisms $w\colon W\rightarrow V$ in $\finite_{\basecat}$.
\end{defn}

\begin{example}
    Consider the case $\basecat=\orbit_G$ for a finite group $G$ and the map $w\colon W \rightarrow V$ is the unique map $G/H\rightarrow G/G$ in $\orbit_G$. Computing the parametrised colimit $\Sigma^w(-)$ in topological $G$-spaces one finds that it is given by smashing with the representation sphere $S^{\widetilde{\mathbb{R}[G/H]}}$ of the reduced permutation representation, see \cite[Ex. 1.18]{dottoMoi} for example.
\end{example}

The reader may want to compare the definition above with \cite[Prop. 1.4.2.11]{lurieHA} to get a feeling for how it will become relevant for our excisiveness considerations later.

\begin{example}
    Let $\udl{\sC}$ be a pointed parametrised bicomplete category equipped with a symmetric monoidal structure which preserves parametrised colimits in each variable. Suppose that the object $S^w\coloneqq \Sigma^w\unit\in\udl{\sC}$ is invertible. Then $\udl{\sC}$ is $w$--spherically invertible. One can show that under our assumptions on the symmetric monoidal structure on $\udl{\sC}$ there is a natural equivalence of endofunctors $\Sigma^w(-)\simeq S^w\otimes(-)$ on $\udl{\sC}$. Thus, since $S^w\otimes(-)$ is an automorphism, so is $\Sigma^w(-)$ and by \cref{prop:suspension_loop_adjunction}, the inverse is necessarily given by $\Omega^w(-)$.
\end{example}

To end this subsection, we provide a ``calculus'' for manipulating these spheres. Before stating the following proposition, we establish the following notation: for a map $w\colon W\rightarrow V$ in $\finite_{\basecat}$, we write $w_+$ for the map $W\sqcup V\xrightarrow{w\sqcup \id} V$.

In the third part of the next result, we shall also need the notion of $\basecat$-distributive symmetric monoidal structures as introduced by Nardin in his thesis \cite{nardinThesis}, which was also further treated in \cite{nardinShah}. See also \cite[Recoll. 2.1.31]{kaifNoncommMotives}.

\begin{prop}\label{prop:calculus_of_spheres}
    Let $\udl{\sC}$ be a pointed $\basecat$--cocomplete category. Let $u\colon U \rightarrow V$ and $w\colon W \rightarrow V$ be maps in $\finite_{\basecat}$. Then:
    \begin{enumerate}[label=(\arabic*)]
        \item $\Sigma^{w_+}\simeq \Sigma\Sigma^{w}\colon V^*\udl{\sC}\rightarrow V^*\udl{\sC}$,
        \item $\Sigma^{u\sqcup w}\simeq \Sigma^{u_+}\Sigma^{w} \colon  V^*\udl{\sC}\rightarrow V^*\udl{\sC}$. In particular, we have $ \Sigma^{u\sqcup w_+}\simeq \Sigma^{u_+}\Sigma^{w_+}\colon  V^*\udl{\sC}\rightarrow V^*\udl{\sC}$,
        \item Suppose $\udl{\sC}$ is a pointed $\basecat$--cocomplete $\basecat$--category equipped with a $\basecat$-distributive symmetric monoidal structure. Then $w_{\otimes}S^1\simeq S^{w_+}\in\udl{\sC}$.
    \end{enumerate}
\end{prop}
\begin{proof}
    To prove (1), we first argue that we may write the category $(w_+)_*\udl{\Delta}^1\backslash\distributiveFinalObject\simeq (\udl{\Delta}^1\times w_*\udl{\Delta}^1)\backslash\distributiveFinalObject$ as a pushout. By the fact that colimits of a diagram of posets and fully faithful functors is just given by the strict colimit in 1--categories (cf. for example the argument in \cite[Obs. 3.2.4]{kaifNoncommMotives}), the required pushout  is given by
    \begin{center}
        \begin{tikzcd}
            \{0\}\times (w_*\Delta^1\backslash\distributiveFinalObject)\rar[hook]\dar[hook]\ar[dr,"\ulcorner",phantom]&\{0\}\times w_*\Delta^1\dar[hook]\\
            \Delta^1\times (w_*\Delta^1\backslash\distributiveFinalObject)\rar[hook]&(w_+)_*\Delta^1\backslash\distributiveFinalObject.
        \end{tikzcd}
    \end{center}
    Hence, for a fixed $X\in V^*\udl{\sC}$, writing $\sigma_X\colon (w_+)_*\Delta^1\backslash\distributiveFinalObject\rightarrow V^*\udl{\sC}$ for the diagram whose colimit is $\Sigma^{w_+}X$ (i.e. the diagram given by $X$ at the initial object and $0$ everywhere else), we obtain by \cref{prop:horev_yanovski} a pushout diagram
    \begin{center}
        \begin{tikzcd}
            \colim_{\{0\}\times (w_*\Delta^1\backslash\distributiveFinalObject)}\sigma_X\rar\dar\ar[dr,"\ulcorner",phantom] & \colim_{\{0\}\times w_*\Delta^1}\sigma_X\dar\\
            \colim_{\Delta^1\times (w_*\Delta^1\backslash\distributiveFinalObject)}\sigma_X\rar & \Sigma^{w_+}X.
        \end{tikzcd}
    \end{center}
    But then it is straightforward to see that we have identifications
    \[\colim_{\{0\}\times (w_*\Delta^1\backslash\distributiveFinalObject)}\sigma_X\simeq \Sigma^wX\quad\quad \colim_{\{0\}\times w_*\Delta^1}\sigma_X\simeq 0 \quad\quad \colim_{\Delta^1\times (w_*\Delta^1\backslash\distributiveFinalObject)}\sigma_X\simeq \colim_{\{1\}\times (w_*\Delta^1\backslash\distributiveFinalObject)}\sigma_X\simeq 0.\]
    Thus the colimit of the pushout diagram above is given by $\Sigma\Sigma^wX$, as claimed.

    For (2), we will decompose the punctured cube $(u\sqcup w)_*\Delta^1\backslash\distributiveFinalObject$ and rewrite it as a parametrised colimit indexed by $(u_+)_*\Delta^1\backslash\distributiveFinalObject$. To this end, consider first the diagram $ (u_*\udl{\Delta}^1)_{/-}\times w_*\udl{\Delta}^1 \colon u_*\udl{\Delta}^1 \longrightarrow \udl{\cat}$ sending $a\mapsto (u_*\udl{\Delta}^1)_{/a}\times w_*\udl{\Delta}^1$. The inclusion $w_*\udl{\Delta}^1\backslash\distributiveFinalObject\hookrightarrow w_*\udl{\Delta}^1$ then induces the diagram
    \begin{equation*}
        \overline{Q}\colon (u_+)_*\udl{\Delta}^1\simeq  u_*\udl{\Delta}^1\times \udl{\Delta}^1\longrightarrow\udl{\cat}\quad :: \quad (x,\epsilon)\mapsto \begin{cases}
            (u_*\udl{\Delta}^1)_{/a}\times (w_*\udl{\Delta}^1\backslash\distributiveFinalObject) & \text{ if } \epsilon=0\\
            (u_*\udl{\Delta}^1)_{/a}\times w_*\udl{\Delta}^1 & \text{ if } \epsilon=1.
        \end{cases}
    \end{equation*}
    We claim  that the colimit of the restricted diagram
    \[Q \colon (u_+)_*\udl{\Delta}^1\backslash\distributiveFinalObject \hookrightarrow (u_+)_*\udl{\Delta}^1 \xlongrightarrow{\overline{Q}}\udl{\cat}\]
    is precisely the punctured cube $(u\sqcup w)_*\Delta^1\backslash\distributiveFinalObject$. To see this, first note that all the maps encoded by the diagram $Q$ are fully faithful, which is a consequence of the fact that for an object $a$ in a poset $\udl{P}$, the map $\udl{P}_{/a}\rightarrow\udl{P}$ is fully faithful because being sliced over an object is just a condition in posets. Secondly,  by \cref{prop:colimtsInCofreeCategories},  the colimit of $Q$ may be computed fibrewise as the ordinary colimit in $\cat$. But again by \cite[Obs. 3.2.4]{kaifNoncommMotives}, the colimit of a diagram of posets along fully faithful functors is just the poset given by the union of posets, which may thus be viewed as a full subcategory of $(u\sqcup w)_*\udl{\Delta}^1$. To conclude, just note that $(u\sqcup w)_*\Delta^1\backslash\distributiveFinalObject$ is fibrewise given by the subposet of $(u\sqcup w)_*\Delta^1\simeq u_*\udl{\Delta}^1\times w_*\udl{\Delta}^1$ on those tuples $(a,b)$ where either $a$ is not the final object in $u_*\udl{\Delta}^1$ or $b$ is not the final object in $w_*\udl{\Delta}^1$. This subposet is clearly precisely the one covered by the colimit of $Q$.
    
    Given the claim, note first that since $(u_*\udl{\Delta}^1)_{/a}$ has a final object, for any category $\udl{J}$, the colimit functor $\udl{\func}((u_*\udl{\Delta}^1)_{/a}\times \udl{J},V^*\udl{\sC})\rightarrow V^*\udl{\sC}$ is equivalently computed as the colimit functor $\udl{\func}(\{a\}\times \udl{J}, V^*\udl{\sC})\rightarrow V^*\udl{\sC}$. For a fixed $a\in u_*\udl{\Delta}^1$ and $X\in V^*\udl{\sC}$, writing $\sigma_X\in\udl{\func}((u\sqcup w)_*\udl{\Delta}^1\backslash\distributiveFinalObject, V^*\udl{\sC})$ for the diagram defining $\Sigma^{u\sqcup w}X$, the restrictions of $\sigma_X$ to $(u_*\udl{\Delta}^1)_{/a}\times w_*\udl{\Delta}^1$ and $(u_*\udl{\Delta}^1)_{/a}\times (w_*\udl{\Delta}^1\backslash\distributiveFinalObject)$ thus have the following colimits:
    \begin{equation*}
        \colim_{(u_*\udl{\Delta}^1)_{/a}\times \udl{J}}\sigma_X \simeq \begin{cases}
            \Sigma^wX & \text{ if } a=\distributiveInitialObject \text{ and } \udl{J}= w_*\udl{\Delta}^1\backslash\distributiveFinalObject\\
            0 & \text{ else}.
        \end{cases}
    \end{equation*}
    Therefore, using the decomposition of the punctured cube above, we may factor the colimit functor $\udl{\func}((u\sqcup w)_*\udl{\Delta}^1\backslash\distributiveFinalObject, V^*\udl{\sC})\rightarrow  V^*\udl{\sC}$ by virtue of \cref{prop:horev_yanovski} as 
    \[\udl{\func}((u\sqcup w)_*\udl{\Delta}^1\backslash\distributiveFinalObject, V^*\udl{\sC})\simeq \lim_{(u_+)_*\udl{\Delta}^1\backslash\distributiveFinalObject}\udl{\func}(Q, V^*\udl{\sC}) \xlongrightarrow{\colim}\lim_{(u_+)_*\udl{\Delta}^1\backslash\distributiveFinalObject} V^*\udl{\sC}\xlongrightarrow{\colim} V^*\udl{\sC}\]
    to compute $\Sigma^{u\sqcup w}X$ as $\Sigma^{u_+}\Sigma^wX$, as required.

    Finally, for (3), since we have a cofibre sequence $S^0\rightarrow0\rightarrow S^1$, applying $w_{\otimes}$ gives us a cofibre sequence
    \[\colim_{w_*\Delta^1\backslash\distributiveFinalObject} w_{\otimes}(S^0\rightarrow 0)\longrightarrow0\longrightarrow w_{\otimes}S^1\]
    by virtue of \cite[Prop. 3.2.8]{kaifNoncommMotives}, whence the claimed equivalence as desired.
\end{proof}

\subsection{Cubical excision}\label{subsec:cubical_excision}

We show now that the parametrised cubes fit into the paradigm of \cref{section:excision_for_posets}. Let $\basecat$ be atomic orbital.

\begin{lem}
    Let $w\colon W \rightarrow V$ be a morphism in $\finite_{\basecat}$. Then the  cube $w_*\udl{\Delta}^1\in\cat_{\basecat_{/V}}$ is a  complementable parametrised poset.
\end{lem}
\begin{proof}
    To see that $w_*\udl{\Delta}^1$ admits finite (co)products, just note that the diagonal functor $\Delta\colon \udl{\Delta}^1\rightarrow \udl{\Delta}^1\times\udl{\Delta}^1 $ admits both left and right adjoints since $\udl{\Delta}^1$ admits finite (co)products. By for instance \cite[Lem. 4.3.2]{kaifPresentable}, we see that $w_*\Delta\colon w_*\udl{\Delta}^1\rightarrow w_*(\udl{\Delta}^1\times\udl{\Delta}^1)\simeq w_*\udl{\Delta}^1\times w_*\udl{\Delta}^1 $ also admits both left and right adjoints, whence the existence of finite (co)products in $w_*\udl{\Delta}^1$. Now, ordinary cubes (i.e. finite products of $\Delta^1$ in $\cat$) are distributive and complementable. Since $w_*\udl{\Delta}^1$ are fibrewise cubes (i.e. fibrewise given by some finite product of $\Delta^1$'s) by \cref{prop:parametrisedCubesAreFibrewiseCubes} and (co)products are given by fibrewise (co)products by \cite[Prop. 5.8]{shahThesis}, we thus see that $w_*\udl{\Delta}^1$ is also distributive and  complementable. This completes the proof. 
\end{proof}

\begin{prop}\label{prop:singletons_subsets_are_excisable_poset}
    Let $w\colon W \rightarrow V$ be a morphism in $\finite_{\basecat}$. Then the full subcategories $\sigma\colon (w_!\terminalTCat)\tcone\subseteq w_*\udl{\Delta}^1$ and ${(w_*\udl{\Delta}^1)}\backslash\distributiveFinalObject\subseteq w_*\udl{\Delta}^1$ are downward--closed. In particular, these two subcategories provide excisable poset structures to $w_*\udl{\Delta}^1$.
\end{prop}
\begin{proof}
    The case of ${(w_*\udl{\Delta}^1)}\backslash\distributiveFinalObject\subseteq w_*\udl{\Delta}^1$ is clear since this is fibrewise given by punctured cubes away from the final objects. To deal with the case $\sigma\colon (w_!\terminalTCat)\tcone\subseteq w_*\udl{\Delta}^1$, by restricting if necessary (and using that singleton inclusions are stable under basechange by \cref{fact:basics_of_singleton_inclusion} (1), suppose we have a morphism in $w_*\udl{\Delta}^1$ with target in $(w_!\terminalTCat)\tcone$. We would like to show that the target lies in $(w_!\terminalTCat)\tcone$ as well. This may be encoded as the dashed lifting problem in the left diagram
    \begin{center}
        \begin{tikzcd}
            \udl{\Delta}^0\dar["\target"'] \rar & (w_!\terminalTCat)\tcone\dar[hook, "\sigma"] && \udl{\Delta}^0\dar["\target"'] \rar & (w^*w_!\terminalTCat)\tcone\dar[hook, "\overline{\sigma}"]\\
            \udl{\Delta}^1\rar\ar[ur, dashed] & w_*\udl{\Delta}^1 && \udl{\Delta}^1\rar\ar[ur, dashed] & \udl{\Delta}^1
        \end{tikzcd}
    \end{center}
    By the $w^*\dashv w_*$ adjunction, this is equivalent to the lifting problem in the right diagram. Now by construction of the map $\sigma$ and using the notation of \cref{eqn:singletonInclusionPullback} for the atomic orbital pullback, the map $\overline{\sigma}\colon (w^*w_!\terminalTCat)\tcone \simeq (\terminalTCat\sqcup \overline{c}_!c^*\terminalTCat)\tcone\rightarrow \udl{\Delta}^1$ is given by the maps $\terminalTCat \xrightarrow{1} \udl{\Delta}^1$, $\overline{c}_!c^*\terminalTCat\xrightarrow{0}\udl{\Delta}^1$, and $\distributiveInitialObject\mapsto 0$. Thus, by a simple case analysis, it is easy to see via this description that the dashed lift in the right square above always exists, whence the dashed lift in the left square as desired.
\end{proof}

\begin{example}
    By virtue of \cref{prop:singletons_subsets_are_excisable_poset}, we obtain the two types of excisable posets that will play a role in the sequel. Let $w\colon W \rightarrow V$ be a morphism in $\finite_{\basecat}$.
    \begin{enumerate}[label=(\alph*)]
        \item We call the tuple of data \[\varphi_w\coloneqq \big(w_*\Delta^1,\:\: \subPoset{(w_*\Delta^1)}\coloneqq (w_!\ast)\tcone,\:\: \varphi\colon (w_!\ast)\tcone\subseteq w_*\Delta^1\big)\] the \textit{singleton $w$--cube.}
        \item We call the tuple of data \[\sigma_w\coloneqq \big(w_*\Delta^1,\:\: \subPoset{(w_*\Delta^1)}\coloneqq w_*\Delta^1\backslash\distributiveFinalObject,\:\:  w_*\Delta^1\backslash\distributiveFinalObject\subseteq w_*\Delta^1\big)\] the \textit{spherical $w$--cube,} the name being justified by \cref{cor:parametrisedSuspensions}.
    \end{enumerate}
\end{example}

\begin{defn}\label{defn:spherical_and_singleton_excisive}
    Let $w\colon W \rightarrow V$ be a morphism in $\finite_{\basecat}$, $\underline{\sC},\underline{\D}\in\cat_{\basecat}$ categories having appropriate parametrised (co)limits, and $F \colon \underline{\sC}\rightarrow \underline{\D}$ a functor. We say that $F$ is \textit{singleton $w$--excisive (resp. spherically $w$--excisive)} if the morphism $V^*F\colon V^*\udl{\sC}\rightarrow V^*\udl{\D}$ in $\cat_{\basecat_{/V}}$ is $\varphi_w$--excisive (resp. $\sigma_w$--excisive), in the sense of \cref{defn:excisive_functors}. 
\end{defn}

The following simple observation will be crucial for our purposes. 

\begin{obs}\label{obs:spherical_excisive_implies_singleton_excisive}
    A spherically $w$--excisive functor is automatically also singleton $w$--excisive. This is an immediate inspection of \cref{defn:excisive_functors}, as well as the fact that the inclusion $\varphi\colon (w_!\ast)\tcone\subseteq w_*\Delta^1$ factors through the inclusion $w_*\Delta^1\backslash\distributiveFinalObject\subseteq w_*\Delta^1$.
\end{obs}

\begin{rmk}\label{rmk:relation_to_dotto}
    In the case of $\basecat=\orbit_G$, a precise comparison between our notions and those of the equivariant Goodwillie calculus developed in \cite{dottoExcision} is complicated by the fact that they are built on top of different models of parametrised $\orbit_G$-categories. However, taking a comparison between these foundations for granted, we expect that Dotto's notion of excisiveness is different to the two notions of excisiveness defined above but is nevertheless expressible in our language.
    
    Namely, in the case where $w\colon W \rightarrow G/G $ is the unique $G$--equivariant map where $W = \coprod_iG/H_i$ is a finite $G$--set, he defines in \textsection 2.1 of \textit{op.~cit.~}the \textit{star subcategory} $\udl{\mathrm{St}}(W_+)\subseteq (w_{+})_*\udl{\Delta}^1$ as the $G$--subposet spanned by the proper subsets of $(G/H_{i})_+$ for all $i$. Comparing \cite[Defs. 2.1 and 2.10]{dottoExcision}, we expect a functor $F\colon \udl{\sC}\rightarrow \udl{\D}$ is $W$--excisive in Dotto's sense if it is excisive in our sense using the excisable poset structure given by $\udl{\mathrm{St}}(W_+)\subseteq (w_{+})_*\udl{\Delta}^1$.  Note importantly that this really depends on the orbit decomposition of $W$ as a $G$--set. Nonetheless, he shows in \cite[Prop. 2.34]{dottoExcision} that this notion still enjoys certain compatibilities with restrictions to subgroups of $G$.

    Dotto's notion is related to the ones in \cref{defn:spherical_and_singleton_excisive} as follows: there are clearly inclusions $\left((w_+)_!\ast\right)\tcone\subseteq \udl{\mathrm{St}}(W_+)\subseteq (w_+)_*\udl{\Delta}^1\backslash \distributiveFinalObject \subseteq (w_+)_*\udl{\Delta}^1$. Thus, as in \cref{obs:spherical_excisive_implies_singleton_excisive}, we see that spherical $w_+$--excisiveness implies Dotto--excisiveness implies singleton $w_+$--excisiveness. We shall see in \cref{rmk:dotto_excision_collapse} that all these notions collapses to the same condition when the target $\udl{\D}$ is equivariantly stable.
\end{rmk}

\subsection{Semiadditivity and stability}\label{subsec:semiadditivity_stability}
In this final subsection, we put all the ingredients together and explain how to use singleton $w$--excisiveness to encode $w$--semiadditivity in the key \cref{keylem:excisive_implies_semiadditivity_induction}. We then use \cref{obs:spherical_excisive_implies_singleton_excisive} to prove the remaining two main theorems. 

Before setting out on the work in earnest, let us remind the reader of the following standard observation from Goodwillie calculus as motivation and also as an ingredient we shall need later.

\begin{obs}\label{obs:1-excisive_semiadditive}
    Let $\sC,\D$ be pointed categories with finite (co)products and $F\colon \sC\rightarrow\D$ be a reduced functor, i.e. it preserves the zero object. If it is 1--excisive, then it is a semiadditive functor, i.e. the canonically constructed map
    \[F(X\sqcup Y) \longrightarrow F(X)\times F(Y)\] is an equivalence. This is because the left square in 
    \begin{center}
        \begin{tikzcd}
            X\sqcup Y \rar\dar & X \dar && F(X\sqcup Y) \rar\dar & F(X) \dar\\
    Y\rar & 0 && F(Y)\rar & 0
        \end{tikzcd}
    \end{center}
    is a pushout in $\sC$ (this is standard, but see also \cref{setting:pointedReducedExcisive} for the more general fact in the parametrised setup). Hence, by 1--excisiveness of $F$, the right square is a pullback in $\D$, i.e. $F(X\sqcup Y)\simeq F(X)\times F(Y)$ as claimed.
\end{obs}

\begin{setting}\label{setting:pointedReducedExcisive}
    Suppose $\basecat$ has a final object $V$ and $w\colon W \rightarrow V$ is the unique map. Suppose $\underline{\sC}$ is pointed, admits finite indexed (co)products, and suppose  we have $\partial \colon \underline{\Delta}^1\rightarrow w^*\underline{\sC}$ given by $(X\rightarrow \ast)$. Note that this is a singleton (co)cartesian diagram. Suppose we have a functor $F\colon \underline{\sC}\rightarrow \underline{\D}$ which is $w$--excisive and preserves zero objects. 

    From these, we can extract the two composites
\begin{equation}\label{eqn:cartesian_and_cocartesian_cubes}
    \alpha\colon w_*\underline{\Delta}^1\xlongrightarrow{w_*\partial} w_*w^*\underline{\sC} \xlongrightarrow{w_!}\underline{\sC}\quad\quad \beta\colon w_*\underline{\Delta}^1\xlongrightarrow{w_*\partial} w_*w^*\underline{\sC} \xlongrightarrow{w_*}\underline{\sC}.
\end{equation}
By \cref{prop:productsOfStronglyCocartesianCubes} and its dual, these are a singleton cocartesian $w$--cube and a singleton cartesian $w$--cube, respectively. Note that $\alpha(\distributiveInitialObject)\simeq w_!X$, $\alpha(\distributiveFinalObject) = w_!\ast \simeq \ast$, $\beta(\distributiveInitialObject)\simeq w_*X$, and $\beta(W)\simeq w_*\ast\simeq \ast$. Here, the equivalences $w_!\ast\simeq \ast\simeq w_*\ast$ is by the pointedness of $\underline{\sC}$. 
\end{setting}

\begin{obs}\label{obs:explicit_f-points_in_w_cubes}
    Let us unwind what the cube $\alpha\colon w_*\underline{\Delta}^1\rightarrow\underline{\sC}$ looks like pointwise. Similar considerations also hold for the cube $\beta$. Consider the value of $\alpha$ at the  $f$--point $x\in w_*\underline{\Delta}^1$. More precisely, consider
    \[f_!\terminalTCat\xlongrightarrow{x} w_*\underline{\Delta}^1\xlongrightarrow{w_*\partial} w_*w^*\underline{\sC}\xlongrightarrow{w_!}\underline{\sC}. \] To work this out more explicitly, consider the pullback 
    \begin{equation}\label{eqn:pullback_beck_chevalley}
        \begin{tikzcd}
            \coprod_{i\in I}Z_i\rar["\sqcup_i{g}_i"]\dar["\sqcup_iu_i"']\ar[dr,phantom, very near start, "\lrcorner"]& W\dar["w"]\\
            F \rar["f"] & T
        \end{tikzcd}
    \end{equation}
    which yields the equivalence $f^*w_*\simeq \prod_iu_{i*}g_i^*$. Thus by adjointing over the $f$--point datum, we obtain
    \[\terminalTCat\xlongrightarrow{\overline{x}} f^*w_*\underline{\Delta}^1\simeq \prod_iu_{i*}g_i^*\underline{\Delta}^1\xlongrightarrow{\prod_iu_{i*}g_i^*\partial} \prod_{i}u_{i*}u_i^*f^*\underline{\sC}\xlongrightarrow{\sqcup_iu_{i!}}f^*\underline{\sC}. \] Since the map $\terminalTCat\xlongrightarrow{\overline{x}} f^*w_*\underline{\Delta}^1\simeq \prod_iu_{i*}g_i^*\underline{\Delta}^1$ chooses either $0$ or $1$ in each of the $i$ factors, we thus see that $\alpha$, when evaluated at the $f$--point $x$, gives
    \[\coprod_{j\in J}u_{j!}g_j^*X\] where $J\subseteq I$ is a subset.
\end{obs}

\begin{example}
Visually, in the nonequivariant case of a 3--cube, the map $\alpha$ for the object $(A,B,C)\in \sC^3$ from \cref{eqn:cartesian_and_cocartesian_cubes} encodes the diagram
\begin{center}
    \begin{tikzcd}
        & C \ar[rr] && \ast\\
        A\vee C \ar[rr] \ar[ur]&& A\ar[ur]\\
        & B\vee C \ar[rr]\ar[uu]&& B\ar[uu]\\
        A\vee B\vee C \ar[rr]\ar[ur]\ar[uu] && A\vee B\ar[uu]\ar[ur]
    \end{tikzcd}
\end{center}
which one can check by hand is singleton cocartesian. For the map $\beta$, one replaces all the $\vee$ with $\times$ in the diagram, and one can also check by hand that it is singleton cartesian.
\end{example}

The two cubical diagrams from \cref{eqn:cartesian_and_cocartesian_cubes} are related by the semiadditivity norm map, as we construct next.

\begin{cons}[Norm transformation on cubes]\label{cons:norm_transformation}
    We construct a transformation $\aleph\colon \alpha\Rightarrow \beta$ which is a morphism in $\udl{\func}(w_*\underline{\Delta}^1,\underline{\sC})$ which we will argue is pointwise given by the appropriate norm transformations. The typical construction of the norm map may be done internally to $\basecat$, giving a norm transformation $w_!\Rightarrow w_*\colon w_*w^*\underline{\sC}\rightarrow \underline{\sC}$, we may define $\aleph$ simply as the transformation
    \begin{center}
        \begin{tikzcd}
            w_*\underline{\Delta}^1 \rar["w_*\partial"] & w_*w^*\underline{\sC}\ar[rrr, bend left = 30, "w_!"description] \ar[rrr,bend left = -30, "w_*"'description]\ar[rrr,phantom, "\Downarrow \text{ norm}"] &&& \underline{\sC}.
        \end{tikzcd}
    \end{center}

    Given this construction, we are left to argue that it is pointwise given by the usual norm maps. So let us consider an arbitrary $f$--point in the diagrams $\alpha$ and $\beta$. We are thus considering the datum 
    \begin{equation}\label{eqn:f_point_unsimplified}
        \begin{tikzcd}
            f_!\terminalTCat\rar & w_*\underline{\Delta}^1 \rar["w_*\partial"] & w_*w^*\underline{\sC}\ar[rrr, bend left = 30, "w_!"description] \ar[rrr,bend left = -30, "w_*"'description]\ar[rrr,phantom, "\Downarrow  \text{ norm}"] &&& \underline{\sC}.
        \end{tikzcd}
    \end{equation}
    which, by adjunction, is equivalently the datum of a morphism in $f^*\underline{\sC}$. To make this more explicit, consider the pullback \cref{eqn:pullback_beck_chevalley}.   Hence, we get an equivalence $f^*w_*\simeq \prod_iu_{i*}g_i^*$. Thus, since the norm maps are stable under basechange, we obtain by adjointing $f_!$ in \cref{eqn:f_point_unsimplified} the diagram
    \begin{center}\small
        \begin{tikzcd}
            \terminalTCat\rightarrow f^*w_*\underline{\Delta}^1 \rar["f^*w_*\partial"] & f^*w_*w^*\underline{\sC}\rar["\simeq"] & \prod_iu_{i*}g_i^*w^*\underline{\sC}\simeq\prod_iu_{i*}u_i^*f^*\underline{\sC}\ar[rrr, bend left = 30, "\sqcup_iu_{i!}"description] \ar[rrr,bend left = -30, "\times_iu_{i*}"'description]\ar[rrr,phantom, "\Downarrow  \text{ norm}"] &&& f^*\underline{\sC}.
        \end{tikzcd}
    \end{center}
    Therefore, we indeed see that evaluating the transformation $\aleph\colon\alpha\Rightarrow \beta$ at the given $f$--point provides the norm map.

    Furthermore, given a functor $F\colon \underline{\sC}\rightarrow \underline{\D}$, we obtain a canonical transformation
    \begin{equation}\label{eqn:norm_map_and_colimit_interchange_on_cubes}
        F\alpha \xlongrightarrow{F\aleph} F\beta \xlongrightarrow{\text{ limit interchange }} \beta_F
    \end{equation}
    of functors $w_*\underline{\Delta}^1\rightarrow \underline{\D}$.
\end{cons}

\begin{prop}\label{prop:equivalence_of_cartesian_diagrams_away_from_top}
Suppose $F\Rightarrow G\colon w_* \Delta^1 \rightarrow \udl{\sC}$ is a natural transformation of  two cartesian diagrams such that $F(\distributiveFinalObject)\simeq G(\distributiveFinalObject)$ and $s^*F\simeq s^* G\colon s^*w_*\Delta^1\backslash\distributiveInitialObject\rightarrow s^*\udl{\sC}$ for all maps $s\colon U\rightarrow T$ which are nonequivalences. Then $F$ is naturally equivalent to $G$.
\end{prop}
\begin{proof}
Because $w_*\Delta^1\simeq (w_*\Delta^1\backslash\distributiveInitialObject)\tcone$, we find that both $F$ and $G$ are right Kan-extended from this full subcategory. In particular it suffices to identify the two functors on this subcategory. This is precisely guaranteed by the hypotheses since the only global objects on $w_*\Delta^1$ are $\distributiveInitialObject$ and $\distributiveFinalObject$, and we have assumed that $F$ and $G$ agree on $\distributiveFinalObject$ as well as on the lower fibres. 
\end{proof}

To prepare for our inductive proofs, we make some general observations about atomic orbital categories.

\begin{defn}
    Let $\basecat$ be an atomic orbital category. We say that \textit{$\basecat$ is of length $n\in \mathbb{N}_{\geq0}\cup \{\infty\}$} if the longest chain of nonequivalent morphisms in $\basecat$ is  $n$, i.e. if we have morphisms $a_1\xrightarrow{f_1}a_2\xrightarrow{f_2}\cdots \xrightarrow{f_i} a_{i+1}$ which are all not equivalences, then $i\leq n$ and there exists a length $n$ chain of nonequivalences. We say that it has \textit{finite length} if it has length $n$ for some $n<\infty$.
\end{defn}

\begin{obs}\label{obs:chain_length_induction}
    Note that if $\basecat$ has a final object and has  length  $0$, then $\basecat\simeq \ast$. Furthermore, note that if $f\colon U \rightarrow W$ is a morphism in $\basecat$ which is not an equivalence, then since the forgetful functor $f\colon \basecat_{/U}\rightarrow\basecat_{/W}$ is conservative, writing $\ell_U, \ell_W$ for the lengths of $\basecat_{/U}$ and $\basecat_{/W}$ respectively, we have $\ell_U \leq \ell_W-1$. This allows us to perform inductive arguments based on the chain lengths if they are all finite.
\end{obs}

\begin{defn}\label{defn:locally_short_categories}
    Let $\basecat$ be an atomic orbital category. We say that it is \textit{locally short} if for every object $V\in\basecat$, the slice category $\basecat_{/V}$ has finite length.    
\end{defn}

Notice that a locally short atomic orbital category which has a final object has finite length.

\begin{rmk}
We claim that given two two objects $X,Y\in \basecat_{/V}$, if there exist maps $f\colon X\rightarrow Y$ and $g\colon Y\rightarrow X$ in $\basecat_{/V}$ then $f$ and $g$ are equivalences.
Namely, we may consider the endomorphism $fg\colon X\rightarrow X$ of $X$. Applying \cite[Lemma 3.37]{CLLPartial} to $\basecat_{/V}$ we find that $fg$ is an equivalence. In particular, $(fg)^{-1} g$ is a section of $f$. But then $f$ (and so also $g$) is an equivalence, by \cite[Lemma 4.3.2]{CLLGlobal}.

In particular we conclude that the relation $X\leq Y$ if and only if $\hom_{\basecat_{/V}}(X,Y) \neq \distributiveInitialObject$ defines a well-defined partial order on the isomorphism classes of $\basecat_{/V}$. We note that an atomic orbital category $\basecat$ is locally short if and only if the partial order on $\pi_0(\basecat_{/V})$ has finite length. 
\end{rmk}

\begin{example}\label{example:atomic_orbital}
    Examples of atomic orbital categories with finite lengths are the orbit category $\orbit_G$ of a finite group $G$, the category $\mathrm{Epi}_d$ of finite sets and surjections up to size $d$, and the category $\mathrm{Vect}^{\mathbf{k}}_{d}$ of $\mathbf{k}$--vector spaces of dimension at most $d$ and surjections, where $\mathbf{k}$ is a finite field.
    
    An example of a locally short atomic orbital category which is itself not of finite length is $\mathrm{Orb}$, the (2,1)-category of finite connected 1--groupoids and faithful maps (i.e. maps that induce injections on $\pi_1$). Another generic example is given by the proper orbit category of an infinite discrete group $G$, that is, the full subcategory of the orbit category on the finite subgroups. 
    
    Finally, the category $\mathrm{Epi}$ of finite sets and surjections of arbitrary size is an an example of an atomic orbital category which is not locally short.
\end{example}

We may now return to the problem at hand. To state it, let us recall the following definition of Nardin \cite{nardinExposeIV}.

\begin{defn}
    Let $\udl{\sC}, \udl{\D}$ be pointed $\basecat$--categories where $\udl{\sC}$ admits finite indexed coproducts and $\udl{\D}$ admits finite indexed products. A functor $F\colon \udl{\sC}\rightarrow \udl{\D}$ is said to be \textit{$\basecat$--semiadditive} if it preserves the zero objects, and  for all maps $w\colon W \rightarrow  V$ in $\finite_{\basecat}$,  the norm transformations $Fw_!\Rightarrow w_*F$ of functors $w_*W^*\udl{\sC}\rightarrow V^*\udl{\D}$  as in \cref{cons:norm_transformation} are equivalences. If $\udl{\sC}$ is pointed and admits finite indexed (co)products, we say that it is \textit{$\basecat$--semiadditive} if  the functor $\id_{\udl{\sC}}\colon\udl{\sC}\rightarrow \udl{\sC}$ is $\basecat$-semiadditive.
\end{defn}

The following is the key lemma.

\begin{lem}\label{keylem:excisive_implies_semiadditivity_induction}
    Suppose $\basecat$ is a locally short atomic orbital category. Let $\underline{\sC}, \udl{\D}$  be  pointed $\basecat$--categories which are parametrised bicomplete. If a reduced functor $F\colon\udl{\sC}\rightarrow\udl{\D}$ is singleton $w$--excisive for all morphisms $w\colon W \rightarrow V$ in $\finite_{\basecat}$, then it is $\baseCat$--semiadditive.
\end{lem}
\begin{proof}
    Since $w$--excisiveness and $w$--semiaddivity are both local conditions, by replacing $\basecat$ with $\basecat_{/V}$ we may assume without loss of generality that $\basecat$ has a final object and is of finite length, and $w$ is a map to the final object. By \cref{obs:1-excisive_semiadditive,obs:chain_length_induction}, the case of length $0$ is true by assumption. We next induct on the length.
    
    By $w$--excisiveness of $F$, the  $w$--cube 
$F\alpha\colon w_*\underline{\Delta}^1\xlongrightarrow{w_*\partial} w_*w^*\underline{\sC} \xlongrightarrow{w_!}\underline{\sC}\xlongrightarrow{F}\underline{\D}$
is cartesian. By pointedness of $F$, $(F\alpha)(\distributiveFinalObject)\simeq \ast$. Hence, we claim that, by induction on lengths,  the norm map \cref{eqn:norm_map_and_colimit_interchange_on_cubes} witnesses an equivalence of $F\alpha$ to
\[\beta_F\colon w_*\underline{\Delta}^1\xlongrightarrow{w_*\partial} w_*w^*\underline{\sC} \xlongrightarrow{w_*w^*F} w_*w^*\underline{\D} \xlongrightarrow{w_*} \D\]
when restricted to $w_*\underline{\Delta}^1\backslash\distributiveInitialObject$. To prove the claim, by \cref{obs:top_points_in_cubes}, the only nonterminal points in $w_*\udl{\Delta}^1\backslash\distributiveInitialObject$ are $f$-points for some $f\colon Y\rightarrow V$ not an equivalence. Using the notations from \cref{eqn:pullback_beck_chevalley}, by \cref{obs:explicit_f-points_in_w_cubes}, such a point looks like $\coprod_{j\in J}u_{j!}g_j^*X$ for some subset $J\subseteq I$. By \cref{cons:norm_transformation}, the transformation $F\alpha\Rightarrow \beta_F$ is given by the norm map 
\begin{equation}\label{eqn:two_norm_maps}
F\left(\coprod_{j\in J}u_{j!}g_j^*X\right)\longrightarrow \prod_{j\in J}F\left(u_{j!}g_j^*X\right) \longrightarrow \prod_{j\in J}u_{j*}g_j^*FX
\end{equation} when evaluated on the $f$-point of $w_*\udl{\Delta}^1\backslash\distributiveInitialObject$ under consideration. By \cref{rmk:classical_cubes,obs:1-excisive_semiadditive}, the singleton excisiveness hypothesis guarantees that the first map in \cref{eqn:two_norm_maps} is an equivalence. On the other hand, the norm map $F\left(u_{j!}g_j^*X\right)\rightarrow u_{j*}g_j^*FX$ is also an equivalence for each $j\in J$ by induction on lengths: by \cref{obs:chain_length_induction}, the length of $\basecat_{/Y}$ is strictly shorter than that of $\basecat_{/V}$ because $f\colon Y\rightarrow V$ was not an equivalence, and so by induction hypothesis, the aforementioned norm map is an equivalence. This completes the proof of the claim.

But then now by \cref{setting:pointedReducedExcisive}, the diagram $\beta_F$  is cartesian. Since $F\alpha$ is cartesian, we see by \cref{prop:equivalence_of_cartesian_diagrams_away_from_top}  that
$Fw_!X \rightarrow w_*w^*FX$ is an equivalence also, as desired. 
\end{proof}

% Finally, when $\basecat$ is locally short, since $w_*\udl{\Delta}^1$ only has global points $\distributiveInitialObject$ and $\distributiveFinalObject$ by \cref{obs:top_points_in_cubes}, the same proof as in the previous paragraph works where the inductive step now reduces to the case of finite length atomic orbital categories with final objects, which have been dealt with above.

\begin{rmk}
For simplicity we have stated the previous result for $\udl{C},\udl{D}$ parametrised bicomplete. However as the proof makes clear, if one takes more care it is possible to prove the statement under weaker assumptions on $\udl{C}$ and $\udl{D}$.
\end{rmk}

\begin{rmk}\label{rmk:applying_theorem_A_indirectly}
    We would like to apply \cref{mainthm:excisive_approximation_left_adjoint} to prove the next theorem using spherical excision with respect to the parametrised cubes $w_*\Delta^1$ where $w\colon W \rightarrow V$ is a map in $\finite_{\basecat}$. Unwinding the formula in \cref{cons:pre-excisive_approximation_functors,cons:excisive_approximation_functors}, we see that if $F\colon\udl{\sC}\rightarrow \udl{\D}$ preserves final objects, then (implicitly basechanging $F$ to a $\basecat_{/V}$--functor),
    \[\goodwillieTFunctor_{\sigma_w}F\simeq\Omega^wF\Sigma^w\quad\quad \goodwillieApprox_{\sigma_w}F\simeq \colim_n\big(F \longrightarrow \Omega^wF\Sigma^w\longrightarrow \Omega^{2w}F\Sigma^{2w}\longrightarrow\cdots\big).\]
    However, we do not know a priori that our spherically faithful category $\udl{\sC}$ will be spherically differentiable in the sense stipulated by \cref{mainthm:excisive_approximation_left_adjoint} since we do not know in general that limits over the punctured   cubes $\mathring{(w_*\Delta^1)}$ commute with sequential colimits. Fortunately, this does not pose a problem since $\Omega^w\colon\udl{\sC}\rightarrow \udl{\sC}$ commuting with sequential colimits suffices: from  the proof of \cref{lem:approximations_are_idempotent}, the only place where differentiability is used is to perform the commutation $\goodwillieApprox_{\sigma}\goodwillieTFunctor_{\sigma}\simeq \goodwillieTFunctor_{\sigma}\goodwillieApprox_{\sigma}$. For this, it suffices that $\Omega^w$ commutes with sequential colimits, since
    \begin{equation*}
        \begin{split}
            \goodwillieApprox_{\sigma_w}\goodwillieTFunctor_{\sigma_w}F&\simeq \colim_n\big(\Omega^wF\Sigma^w\longrightarrow \Omega^{2w}F\Sigma^{2w}\longrightarrow \Omega^{3w}F\Sigma^{3w}\longrightarrow\cdots\big)\\
            &\simeq \Omega^w\colim_n\big(F\Sigma^w\longrightarrow \Omega^{w}F\Sigma^{2w}\longrightarrow \Omega^{2w}F\Sigma^{3w}\longrightarrow\cdots\big)\\
            &\simeq \goodwillieTFunctor_{\sigma_w}\goodwillieApprox_{\sigma_w}F
        \end{split}
    \end{equation*}
    as required. In particular, we may apply \cref{mainthm:excisive_approximation_left_adjoint} to the special case of the identity functor $\id\colon \udl{\sC}\rightarrow\udl{\sC}$ if $\udl{\sC}$ were spherically faithful with sufficient parametrised (co)limits to get that the identity functor is spherically excisive.
\end{rmk}

We are now ready to gather the ingredients and prove \cref{alphthm:spherically_faithful_implies_semiadditivity}.

\begin{thm}\label{mainthm:spherically_faithful_implies_semiadditivity}
Suppose $\basecat$ is a locally short atomic orbital category. Let $\underline{\sC}$  be a pointed $\basecat$--category which is parametrised bicomplete. If $\underline{\sC}$  is spherically faithful, then it is $\baseCat$--semiadditive.
\end{thm}
\begin{proof}
    Let $w$ be a morphism in $\finite_{\basecat}$. Since $\udl{\sC}$ was $w$--spherically faithful, we get that $\id\rightarrow\Omega^w\Sigma^w$ is an equivalence. Thus, by the formula for $\goodwillieApprox_{\sigma_w}$ from \cref{mainthm:excisive_approximation_left_adjoint} together with \cref{rmk:applying_theorem_A_indirectly}, we see that the transformation $\id_{\udl{\sC}}\rightarrow \goodwillieApprox_{\sigma_w}\id_{\udl{\sC}}$ is an equivalence, i.e. that $\id_{\udl{\sC}}\colon \udl{\sC}\rightarrow\udl{\sC}$ is spherically $w$--excisive. Hence, by \cref{obs:spherical_excisive_implies_singleton_excisive}, we see that $\id_{\udl{\sC}}$ is also singleton $w$--excisive. Since $w$ was arbitrary, we get that $\id_{\udl{\sC}}$ is singleton $w$--excisive for all $w$, and so by \cref{keylem:excisive_implies_semiadditivity_induction}, $\udl{\sC}$ is parametrised semiadditive, as wanted.
\end{proof}

This then yields \cref{mainThm:sphere_invertible_equivalent_to_parametrised_stable} almost immediately. Before we state it, let us recall now Nardin's \cite{nardinExposeIV} notion of parametrised stability for atomic orbital base categories $\basecat$: a $\basecat$--category $\udl{\sC}$ is  said to be $\basecat$-stable if it is $\basecat$-semiadditive, admits parametrised fibrewise finite (co)limits, and  $\udl{\sC}(V)$ is a stable category for all $V\in \basecat$.

\begin{thm}\label{mainThm:sphere_invertible_equivalent_to_parametrised_stable}
    Suppose $\basecat$ is a locally short atomic orbital category such that $\myuline{\spectra}$ admits a $\basecat$-distributive symmetric monoidal structure.  Let $\udl{\sC}$ be a pointed parametrised presentable $\basecat$--category.    
    Then $\udl{\sC}$ is parametrised stable if and only if it is spherically invertible. 
\end{thm}
\begin{proof}
    That spherical invertibility implies parametrised stability is the content of \cref{mainthm:spherically_faithful_implies_semiadditivity}, since spherical invertibility implies in particular spherical faithfulness. For the converse, note by \cref{prop:calculus_of_spheres} (1, 3) that $S^w\simeq \Omega w_{\otimes}S^1\in V^*\myuline{\spectra}$ are invertible objects since the norm functors, which exist by hypothesis on $\basecat$, are symmetric monoidal.  Hence, $\Sigma^w\colon V^*\myuline{\spectra}\rightarrow V^*\myuline{\spectra}$ is an equivalence. Since $V^*\udl{\sC}$ is a module over $V^*\myuline{\spectra}$ in $\presentable_{\basecat_{/V}}^L$ by \cite[Prop. 2.2.19]{kaifNoncommMotives} by assumption, we thus see that $\Sigma^w\colon V^*\udl{\sC}\rightarrow V^*\udl{\sC}$ is equivalent to the map $\Sigma^w\otimes\id\colon V^*\myuline{\spectra}\otimes V^*\udl{\sC}\xrightarrow{\simeq} V^*\myuline{\spectra}\otimes V^*\udl{\sC}$, which is an equivalence, as was to be shown.
\end{proof}

\begin{rmk}\label{rmk:explanation_of_nardin_issues}
Both $\myuline{\spectra}$ and the parametrised category $\myuline{\spc}_\ast [(S^w)^{-1}]$ obtained by (fibrewise) inverting the spheres $S^w$ for all $w\colon W\to V$ in $\myuline{\spc}$ are idempotent algebras in parametrised presentable categories, classifying the properties of being parametrised stable and spherically invertible respectively. By \cref{mainthm:spherically_faithful_implies_semiadditivity}, the latter subcategory of presentable categories is contained in the former. Therefore the conclusion of \cref{mainThm:sphere_invertible_equivalent_to_parametrised_stable} holds if and only if the canonical map $\myuline{\spectra} \to \myuline{\spc}_\ast [(S^w)^{-1}]$ is an equivalence. As will be shown in forthcoming work of the first author with Branko Juran and Natalie Stewart, and independently also by Emma Brink and Tobias Lenz, pointed spaces admit a $\basecat$-distributive symmetric monoidal structure. One can show that the spheres $S^w$ are closed under norms, and so $\myuline{\spc}_\ast [(S^w)^{-1}]$  clearly attains a $\basecat$--distributive symmetric monoidal structure. Therefore if the conclusion of the theorem holds, $\myuline{\spectra}$ also admits a $\basecat$--distributive structure. In particular the assumption in \cref{mainThm:sphere_invertible_equivalent_to_parametrised_stable} is necessary. In other words, parametrised stability agrees with spherically invertibility if and only if $\myuline{\spectra}$ admits a $\basecat$-distributive symmetric monoidal structure.

Unfortunately it is unknown to us if this is the case for all locally short atomic orbital categories $\basecat$. This was previously claimed by Nardin \cite[\textsection 3.5]{nardinThesis}, but the proof is currently believed to have gaps and it is unclear if the result is true in full generality. The issues in the proof will be explained in the same forthcoming work of the first author, Branko Juran, and Natalie Stewart referred to before. Luckily, for $\basecat = \mathcal{O}_G$, the category $\myuline{\spectra}$ admits a $\basecat$-distributive symmetric monoidal structure by \cite[Theorem 5.5.1]{normed_rings}. 
\end{rmk}

Finally, having shown the confluence of two natural notions of parametrised stability  on categories under suitable hypotheses in \cref{mainThm:sphere_invertible_equivalent_to_parametrised_stable}, we end the main body of this work by complementing this with a similar confluence of natural notions for ``parametrised linearity'' on functors into stable categories which are spherically faithful. In particular, this extends the validity of \cite[Prop. 3.33]{dottoMoi} to an arbitrary locally short atomic orbital base category $\basecat$.

For this, recall the notation of $w_+\colon W\sqcup V\rightarrow V$  from \cref{prop:calculus_of_spheres} associated to any map $w\colon W\rightarrow V$ in $\finite_{\basecat}$.

\begin{thm}\label{thm:characterisations_of_linearity}
    Let $\baseCat$ be a locally short atomic orbital category, $\underline{\sC}$ a pointed parametrised bicomplete category, $\underline{\D}$ a parametrised  presentable stable category which is spherically faithful, and $F\colon \underline{\sC}\rightarrow\underline{\D}$ a reduced functor preserving filtered colimits. Then the following conditions are equivalent:
    \begin{enumerate}[label=(\arabic*)]
        \item $F$ is $\sigma_{w_+}$--excisive for every map $w\colon W\rightarrow V$ in ${\baseCat}$,
        \item $F$ is $\varphi_{w}$--excisive for every map $w\colon W\rightarrow V$ in $\finite_{\baseCat}$,
        \item $F$ preserves parametrised colimits.
    \end{enumerate}
\end{thm}
\begin{proof}
    Firstly, note that   (1) is equivalent to requiring that $F$ is $\sigma_{w_+}$--excisive more generally for $w\colon W \rightarrow V$ in $\finite_{\basecat}$ by \cref{cons:excisable_structure_on_face_posets} and \cref{prop:excision_implications_along_inclusions_of_orbits}. Also, condition (2)  implies $1$--excisiveness since this may equivalently be rephrased as $\varphi_{\ast_+}$--excisiveness.  Next, by local shortness, we know that (2) implies (3) since $\varphi_w$--excisiveness implies $\basecat$--semiadditivity by \cref{keylem:excisive_implies_semiadditivity_induction} and $1$--excisive functors preserving filtered colimits  into a stable category implies preservation of ordinary  colimits.

    For the rest of the proof, since the notions of excisiveness is local in $\basecat$ so that we will be working with the base category $\basecat_{/V}$ anyway, we may assume that $w$ maps $W$ to the final object and write $F$ instead of $V^*F$.     To see that (1) implies (2), we  argue that $F$ is  $\sigma_w$--excisive for all $w \colon W\rightarrow V$ in $\finite_{\baseCat}$. This will in particular imply that $F$ is also $\varphi_w$--excisive by \cref{obs:spherical_excisive_implies_singleton_excisive}. 
    
    Since $F$ is $\sigma_{w_+}$--excisive, we know that $F\rightarrow P_{\sigma_{w_+}}F  \simeq \colim\big(F\rightarrow \Omega^{w_+}F\Sigma^{w_+}\rightarrow \Omega^{2(w_+)}F\Sigma^{2(w_+)}\rightarrow\cdots\big)$
    is an equivalence. On the other hand, the sequential colimit is also the diagonal sequential colimit in the diagram
    \[
    \begin{tikzcd}
        F \rar\dar & \Omega^wF\Sigma^w \rar\dar & \Omega^{2w}F\Sigma^{2w} \rar\dar & \cdots \\
        \Omega F\Sigma \rar\dar & \Omega^w\Omega F\Sigma\Sigma^w \rar\dar & \Omega^{2w}\Omega F\Sigma\Sigma^{2w} \rar\dar & \cdots \\
        \Omega^2 F\Sigma^2 \rar\dar & \Omega^w\Omega^2 F\Sigma^2\Sigma^w \rar\dar & \Omega^{2w}\Omega^2 F\Sigma^2\Sigma^{2w} \rar\dar & \cdots \\
        \vdots & \vdots & \vdots 
    \end{tikzcd}
    \] where we have also used that $\Omega^{n(w_+)} \simeq \Omega^{nw\sqcup n+} \simeq \Omega^{nw}\Omega^{n}$   and  $\Sigma^{n(w_+)} \simeq \Sigma^{nw\sqcup n+} \simeq \Sigma^{nw}\Sigma^{n}$ from \cref{prop:calculus_of_spheres}. Hence, by siftedness, we may also compute $P_{\sigma_{w_+}}F$ as 
    \begin{equation*}
        \begin{split}
            &\colim\big(P_{\sigma_w}F \rightarrow P_{\sigma_w}\Omega F\Sigma \rightarrow P_{\sigma_w}\Omega^2 F\Sigma^2\rightarrow \cdots\big)\\
            &\simeq P_{\sigma_w}\colim\big(F \rightarrow \Omega F\Sigma \rightarrow \Omega^2 F\Sigma^2\rightarrow \cdots\big)\simeq P_{\sigma_w}P_1F.
        \end{split}
    \end{equation*}
    Here, the first equivalence uses that $\Omega^w\colon \underline{\D}\rightarrow\underline{\D}$ preserves sequential colimits by spherical faithfulness. Hence,  we get $F\simeq P_{\sigma_{w_+}}F\simeq P_{\sigma_w}P_1F$, and so $F$ is in particular $\sigma_w$--excisive, as wanted.

    Finally, to see (3) implies (1), just consider
    \begin{equation*}
        \begin{split}
            P_{\sigma_{w_+}}F & \simeq \colim\big(F\rightarrow \Omega^{w_+}F\Sigma^{w_+}\rightarrow \Omega^{2w_+}F\Sigma^{2w_+}\rightarrow\cdots\big)\\
            &\simeq \colim\big(F\rightarrow \Omega^{w_+}\Sigma^{w_+}F\rightarrow \Omega^{2(w_+)}\Sigma^{2(w_+)}F\rightarrow\cdots\big)\\
            &\simeq \colim\big(F\xrightarrow{=} F\xrightarrow{=} F\xrightarrow{=}\cdots\big) = F\\
        \end{split}
    \end{equation*}
    where the first equivalence is by the formula from \cref{rmk:applying_theorem_A_indirectly} and that $F$ is reduced, the second is by hypothesis that $F$ preserves parametrised colimits, and the third is since $\underline{\D}$ was spherically faithful, and so $\id\rightarrow\Omega^{w_+}\Sigma^{w_+}$ are equivalences by \cref{mainThm:sphere_invertible_equivalent_to_parametrised_stable}. Hence, $F\simeq P_{\sigma_{w_+}}F$, and so it is $\sigma_{w_+}$--excisive.
\end{proof}

\begin{rmk}\label{rmk:dotto_excision_collapse}
    As explained in \cref{rmk:explanation_of_nardin_issues}, in the case of equivariant homotopy theory for finite groups, we may apply \cref{mainThm:sphere_invertible_equivalent_to_parametrised_stable} to get that a $G$--presentable category is $G$--stable if and only if it is spherically invertible. Thus, by \cref{thm:characterisations_of_linearity} and the last paragraph of \cref{rmk:relation_to_dotto}, we see in this case that a reduced, filtered--colimit--preserving functor $F\colon \udl{\sC}\rightarrow\udl{\D}$ with $\udl{\D}$ being $G$--presentable and $G$--stable satisfies furthermore Dotto's excisiveness if and only if it is singleton/spherically excisive if and only if it preserves all $G$--colimits. And so all these notions collapse to a single one in this regime.
\end{rmk}

\appendix

\section{Decomposition of colimit diagrams}\label{section:decomposition_of_colimits}

Suppose $\udl{I}$ is a parametrised category which admits a decomposition into a colimit $\udl{I}\simeq \colim_{\udl{J}} \udl{I_j}$. In this appendix we explain how one can compute colimits over $\udl{I}$ in two steps: first one computs the colimit over each $\udl{I}_j$, and then one computes the colimit of these objects over $\udl{J}$. We will deduce this via the approach suggested by Horev-Yanovski in \cite{horevYanovski}.

\begin{defnAppx}
Suppose $\udl{\sC}_\bullet \colon \udl{I}\to \udl{\cat}_\basecat$ is a functor of $\basecat$-categories. We define $p\colon \udl{\unstraighten}(\udl{\sC}_\bullet)\to \udl{I}$ to be the cocartesian unstraightening of $\udl{\sC}_\bullet$, in the sense of \cite[Rmk~7.2]{shahThesis}.
\end{defnAppx}

\begin{defnAppx}
Suppose $\udl{\sC}_\bullet \colon \udl{I}\to \udl{\cat}_\basecat$ is a functor of $\basecat$-categories. We define the $\basecat$-category $\laxlim(\udl{\sC}_\bullet)$ to be the $\basecat$-category of sections of the unstraightening $p\colon \udl{\unstraighten}(\udl{\sC}_\bullet)\to \udl{I}$. More formally $\laxlim(\udl{\sC}_\bullet)$ is defined via the pullback
\[
\begin{tikzcd}
\laxlim(\udl{\sC}_\bullet) \arrow[r] \arrow[d] & \udl{\func}(\udl{I},\udl{\unstraighten}(\udl{\sC}_\bullet)) \arrow[d, "p_*"] \\
\{\id_{\udl{I}}\} \arrow[r] & \udl{\func}(\udl{I},\udl{I}). 
\end{tikzcd}
\]
\end{defnAppx}

\begin{rmkAppx}
Given an object $A\in \basecat$ and an object $i\in \udl{I}(A)$, the composite 
\[
\udl{A}\xrightarrow{i} \udl{I} \xrightarrow{s} \udl{\unstraighten}(\udl{\sC}_\bullet) 
\]
factors through $\udl{\sC}_{i} \coloneqq \udl{A}\times_{\udl{I}} \udl{\unstraighten}(\udl{\sC}_\bullet)$, the fibre of $\udl{\unstraighten}(\udl{\sC}_\bullet)$ over $i\in \udl{I}$. We will identify $s(i)$ with an object of $\udl{\sC}_i$ in this way. An object of $\laxlim(\udl{\sC}_\bullet)$ can be viewed as a family of such objects parametrised by $\udl{I}$.
\end{rmkAppx}

\begin{rmkAppx}
Observe that, just as in ordinary category theory, $\lim \udl{\sC}_\bullet$ is the full subcategory of $\laxlim(\udl{\sC}_\bullet)$ spanned by the cocartesian sections. For more details see \cite[Thm. 3.4.2]{kaifPresentable}.
\end{rmkAppx}

Note that any natural transformation $F_\bullet\colon \udl{\sC}_\bullet \rightarrow \udl{\sD}_\bullet$ induces a functor $F\colon \laxlim(\udl{\sC}_\bullet)\rightarrow \laxlim(\udl{\sD}_\bullet)$, via post-composition with the unstraightening of $F_\bullet$. Because this preserves cocartesian edges, $F$ preserves cocartesian sections and hence restricts to a functor $F\colon \lim \udl{\sC}_\bullet \rightarrow \lim \udl{\sD}_\bullet$.

\begin{lemAppx}\label{lem:rel_adj}
Suppose $F_\bullet\colon \udl{\sC}_\bullet \rightarrow \udl{\sD}_\bullet$ is a natural transformation such that for all $i\in \udl{I}$, $F_i\colon \udl{\sC}_i\rightarrow \udl{\sD}_i$ is a left adjoint with right adjoint $G_i\colon \udl{\sD}_i\rightarrow \udl{\sC}_i$. Then $F\colon \laxlim(\udl{\sC}_\bullet)\to \laxlim(\udl{\sD}_\bullet)$ admits a right adjoint, which we denote by $G$. Given a section $s\colon \udl{I}\rightarrow \laxlim(\udl{\sD}_\bullet)$ and an object $i\in \udl{I}(A)$, the object $G(s)(i)$ is equivalent to $G_i(s(i))$. Informally, $G$ is given by applying the $G_i$ pointwise to the section $s$. 
\end{lemAppx}

\begin{proof}
As observed in \cite[Rmk. 7.4]{shahThesis}, the $\infty$-category of cocartesian fibrations over $\udl{I}$ can be identified with ordinary cocartesian fibrations over $\totalCategory(\udl{I})$. One may check that under this correspondence, $\laxlim(\udl{\sC}_\bullet)$ corresponds to sections of the cocartesian fibrations of the corresponding functor $\udl{\sC}_\bullet\colon \totalCategory(\udl{I})\rightarrow \cat$ such that the restriction to the fibre of $\totalCategory(\udl{I})$ over any point $A\in \basecat$ is cocartesian. Under these identifications the claimed result is a special case of \cite[Lem. 3.19]{sil_globalize}.
\end{proof}

Now suppose that we are in the special case that $\udl{\sC}_\bullet$ is the constant functor on a $\basecat$-category $\udl{\sC}$. Then a simple calculation shows that $\underline{\unstraighten}(\udl{\sC}_\bullet)$ is given by $\pi_2\colon \udl{\sC}\times \udl{I}\rightarrow \udl{I}$, and so $\laxlim(\udl{\sC}_\bullet)$ is equivalent to $\udl{\func}(\udl{I},\udl{\sC})$. Now suppose we are given a cone $F\colon \udl{\sC}\rightarrow \udl{\sD}_\bullet$ of pointwise left adjoints. The previous result then gives an adjunction
\[
F\colon \func(\udl{I},\udl{\sC}) \leftrightarrows \laxlim(\udl{\sD}_\bullet) \colon G.
\]

We will now explain how to apply this to identify the right adjoint out of a limit.

\begin{thmAppx}
Suppose $\udl{I}$ is a $\basecat$-category, $\udl{\sD}_\bullet\colon \udl{I} \rightarrow \cat_\basecat$ is a diagram of $\basecat$-categories, and $\udl{\sC}\rightarrow \udl{\sD}_\bullet$ is a cone. Assume that for every $i\in \udl{I}$ the functor $F_i\colon \udl{\sC}\rightarrow \udl{\sD}_i$ admits a right adjoint $G_i$, and that $\udl{\sC}$ admits all $\udl{I}$-limits. Then the functor $F\colon \udl{\sC}\rightarrow \lim \udl{\sD}_\bullet$ admits a right adjoint, described as the composite
\[
\lim \udl{\sD}_\bullet \xrightarrow{G} \udl{\func}(\udl{I},\udl{\sC}) \xrightarrow{\lim_{\udl{I}}} \udl{\sC}.
\]
\end{thmAppx}

\begin{proof}
Consider the functor 
\[
\udl{\sC} \xrightarrow{\Delta} \udl{\func}(\udl{I},\udl{\sC}) \xrightarrow{F} \laxlim(\udl{\sD}_\bullet). 
\]
By our assumptions this admits a right adjoint, given by the composite 
\[
\laxlim(\udl{\sD}_\bullet)\xrightarrow{G} \udl{\func}(\udl{I},\udl{\sC}) \xrightarrow{\lim_{\udl{I}}} \udl{\sC}.
\]
We claim this restricts to the adjunction of the statement, for which it suffices to check that the first composite factors through $\lim(\udl{\sD}_\bullet)$. The functor $\Delta$ clearly sends objects to cocartesian sections, because any equivalence is cocartesian. We have previously observes that $F$ does as well, see the discussion before Lemma \ref{lem:rel_adj}, and so the composite restricts as required.
\end{proof}

Note that the previous theorem admits an obvious dual statement, which we apply to immediately obtain the following corollary.

\begin{corAppx}[Decomposition of colimits]\label{prop:horev_yanovski}
Consider a diagram $\udl{I}_\bullet\colon \udl{J}\rightarrow \cat_\basecat$, and write $\udl{I}$ for the colimit of $\udl{I}_\bullet$. Suppose $\udl{\sC}$ is some $\basecat$-category which admits $\udl{J}$-colimits and $\udl{I}_j$-colimits for all $j\in \udl{J}$. Then $\udl{\sC}$ admits $\udl{I}$-colimits and $\colim_{\udl{I}}\colon \func(\udl{I},\udl{\sC})\to \udl{\sC}$ is equivalent to the composite
\[
\udl{\func}(\udl{I},\udl{\sC}) \simeq \lim_{\udl{J}}\udl{\func}(\udl{I}_j,\udl{\sC}) \xrightarrow{\colim_{\udl{I}_j}} \func(\udl{J},\udl{\sC}) \xrightarrow{\colim_{\udl{J}}} \udl{\sC}. \qedhere
\]
\end{corAppx}

\emergencystretch 4em

\printbibliography

\end{document}